\documentclass[10pt]{amsart}
\usepackage{amssymb,amsmath}
\usepackage{appendix}
\usepackage{paralist}
\usepackage{graphics}
\usepackage{epsfig}
\usepackage{graphicx}
\usepackage{epstopdf}
\newtheorem{theorem}{Theorem}[section]
\newtheorem{remark}{Remark}[section]
\newtheorem{lemma}{Lemma}[section]
\newtheorem{proposition}[theorem]{Proposition}

\theoremstyle{definition}

\numberwithin{equation}{section}

\newcommand{\uu}{{\sf u}}
\newcommand{\UU}{{\sf U}}

\title[An IMEX  FE method for a linearized CHC equation]
{An IMEX Finite Element Method\\
for a linearized Cahn-Hilliard-Cook equation\\
driven by the space derivative\\
 of a space-time white noise}
\author[]{Georgios E. Zouraris$^{\dag}$}
\thanks{Work partially supported by The Research Committee of The University
of Crete under Research Grant \#4339:
`{\sl Numerical solution of stochastic partial differential equations}' funded by
The Research Account of the University of Crete (2015-2016).}
\thanks{%
$^{\ddag}$Department of Mathematics and Applied Mathematics,
University of Crete,  P.O. Box 2208, GR--710 03 Heraklion, Crete, Greece}
 \email{georgios.zouraris@uoc.gr}
\newcommand{\rset}{{\mathbb R}}
\newcommand{\nset}{{\mathbb N}}

\newcommand{\bfdot}{\bf\dot}
\newcommand{\dtau}{\Delta\tau}

\newcommand{\ken}{\ \ }

\newcommand{\ssy}{\scriptscriptstyle}

\newcommand{\half}{\frac{1}{2}}

\begin{document}
\maketitle
\centerline{\today}
%
%
%
%
%
%
\begin{abstract}
We consider a model initial- and Dirichlet boundary- value problem for
a linearized Cahn-Hilliard-Cook equation, in one space dimension,
forced by the space derivative of a space-time white noise.
First, we introduce a canvas problem the solution to which is a regular
appro\-ximation of the mild solution to the problem and
depends on a finite number of random variables.
Then, fully-discrete approximations of the solution to the
canvas problem are constructed using, for discretization in
space, a Galerkin finite element method based on $H^2$ piecewise
polynomials, and, for time-stepping, an implicit/explicit method.
Finally, we derive a strong a priori estimate of the error approximating the
mild solution to the problem by the canvas problem solution, and of the
numerical approximation error of the solution to the canvas problem. 
\end{abstract}
%
%
%
\section{Introduction}\label{SECT1}
%
%
%
%
%
%
Let $T>0$, $D:=(0,1)$ and $(\Omega,{\mathcal F},P)$ be a complete
probability space. Then, we consider the model initial-
and Dirichlet boundary- value problem for a linearized
Cahn-Hilliard-Cook equation formulated in \cite{KZ2013b},
which is as follows: find a stochastic function
$u:[0,T]\times{\overline D}\to\rset$ such that
\begin{equation}\label{PARAB}
\begin{gathered}
u_t+u_{xxxx}+\mu\,u_{xx}
=\partial_x{\dot W}(t,x)
\quad\forall\,(t,x)\in (0,T]\times D,\\
u(t,\cdot)\big|_{\ssy\partial D}=u_{xx}(t,\cdot)\big|_{\ssy\partial D}=0
\quad\forall\,t\in(0,T],\\
u(0,x)=0\quad\forall\,x\in D,\\
\end{gathered}
\end{equation}
a.s. in $\Omega$, where ${\dot W}$ denotes a space-time white
noise on $[0,T]\times D$ (see, e.g., \cite{Walsh86},
\cite{KXiong}) and $\mu$ is a real constant.
We recall that the mild solution of the problem above
(cf. \cite{DZ2007}) is given by
\begin{equation}\label{MildSol}
u(t,x)=\int_0^t\!\int_{\ssy D}{\sf\Psi}_{t-s}(x,y)\,dW(s,y),
\end{equation}
where
\begin{equation}\label{GreenKernel}
\begin{split}
{\sf\Psi}_t(x,y):=&\,-\sum_{k=1}^{\infty}\lambda_k\,e^{-\lambda_k^2\,(\lambda_k^2-\mu)
t} \,\varepsilon_k(x)\,\varphi_k(y)\\
=&\,-\partial_y{\sf G}_t(x,y) \quad
\forall\,(t,x,y)\in(0,T]\times{\overline D}\times{\overline D},
\end{split}
\end{equation}
$\lambda_k:=k\,\pi$ for $k\in\nset$,
$\varepsilon_k(z):=\sqrt{2}\,\sin(\lambda_k\,z)$
and $\varphi_k(z):=\sqrt{2}\,\cos(\lambda_k\,z)$
for $z\in{\overline D}$ and $k\in\nset$,
and  ${\sf G}_t(x,y)$ is the space-time Green kernel
of the solution to the deterministic parabolic problem:
find  $w:[0,T]\times{\overline D}\to\rset$ such that
\begin{equation}\label{Det_Parab}
\begin{gathered}
w_t + w_{xxxx}+\mu\,w_{xx}= 0
\quad\forall\,(t,x)\in (0,T]\times D,\\
w(t,\cdot)\big|_{\ssy\partial D}=
w_{xx}(t,\cdot)\big|_{\ssy\partial D}=0
\quad\forall\,t\in(0,T],\\
w(0,x)=w_0(x)\quad\forall\,x\in D.\\
\end{gathered}
\end{equation}
%
%
\par
In the paper at hand, our goal is to propose and analyze a numerical method for the approximation
of $u$ that has less stability requirements and lower complexity than the method proposed in
\cite{KZ2013b}.
\subsection{A canvas problem}\label{The_Reg_Problem}
\par
A {\it canvas} problem is an initial- and boundary- value problem
the solution to which: i) depends on a finite number of random variables and
ii) is a regular approximation of the mild solution $u$ to \eqref{PARAB}. Then,
we can derive computable approximations of $u$ by constructing numerical
approximations of the canvas problem solution via the application of a
discretization technique for stochastic partial differential equations with random
coefficients. The formulation of the canvas problem depends on the way we
replace the infinite stochastic dimensionality of the problem \eqref{PARAB} by a finite one.
\par
In our case the canvas problem is formulated as follows (cf. \cite{ANZ}, \cite{KZ2010}, \cite{KZ2013b}):
Let ${\sf M},{\sf N}\in{\mathbb N}$, $\Delta{t}:=\frac{T}{{\sf N}}$, $t_n:=n\,\Delta{t}$ 
for $n=0,\dots,{\sf N}$, ${\sf T}_n:=(t_{n-1},t_n)$ for $n=1,\dots,{\sf N}$,
and $\uu:[0,T]\times{\overline{D}}\rightarrow{\mathbb R}$ such that
\begin{equation}\label{AC2}
\begin{gathered}
\uu_t +\uu_{xxxx} +\mu\,\uu_{xx}=\partial_x{\mathcal W}
\quad\text{\rm in}\ken (0,T]\times D,\\
\uu(t,\cdot)\big|_{\ssy\partial D}=
\uu_{xx}(t,\cdot)\big|_{\ssy\partial D}=0
\quad\forall\,t\in(0,T],\\
\uu(0,x)=0\quad\forall\,x\in D,\\
\end{gathered}
\end{equation}
where
\begin{equation}\label{SLoad}
{\mathcal W}(\cdot,x)|_{\ssy{\sf T}_n} :=
\tfrac{1}{\Delta{t}}\,
\sum_{i=1}^{\ssy{\sf M}}
R^n_i\,\varphi_i(x)
\quad\forall \,x\in D,\quad n=1,\dots,{\sf N},
\end{equation}
\begin{equation}\label{RDef}
R^n_i:=\int_{\ssy{\sf T}_n}\!\int_{\ssy D}\varphi_i(x)\;dW(t,x)
=B^i(t_{n+1})-B^i(t_n),
\quad i=1,\dots,{\sf M},\quad n=1,\dots,{\sf N},
\end{equation}
and $B^i(t):=\int_0^t\int_{\ssy D}\varphi_i(x)\;dW(s,x)$
for $t\ge0$ and $i\in{\mathbb N}$. According to \cite{Walsh86},
$(B^i)_{i=1}^{\infty}$ is a family of independent Brownian motions,
and thus, the random variables
$\left(\left(R^n_i\right)_{n=1}^{\ssy {\sf N}}\right)_{i=1}^{\ssy\sf M}$
are independent and satisfy
\begin{equation}\label{WRemark}
R^n_i\sim{\mathcal N}(0,\Delta{t}),\quad
i=1,\dots,{\sf M},\quad n=1,\dots,{\sf N}.
\end{equation}
%
%
Thus, the solution $\uu$ to \eqref{AC2} depends on ${\sf N}{\sf M}$ random variables
and the well-known theory for parabolic problems (see, e.g, \cite{LMag}) yields its regularity
along with the following representation formula:
\begin{equation}\label{HatUform}
\begin{split}
\uu(t,x)=&\,\int_0^t\!\!\int_{\ssy D}{\sf G}_{t-s}(x,y)
\,\partial_y{\mathcal W}(s,y)\,dsdy\\
=&\,\int_0^t\!\!\int_{\ssy D}{\sf\Psi}_{t-s}(x,y) \,{\mathcal W}(s,y)
\,dsdy \quad\forall\,(t,x)\in[0,T]\times{\overline D}.\\
\end{split}
\end{equation}
%
%
%
%
%
%
%
\begin{remark}\label{complex_rem}
In \cite{KZ2013b} the definition of ${\mathcal W}$ is based on a
uniform partition of $[0,T]$ in $N$ subintervals 
and on a uniform partition of $D$ in $J$ subintervals.
At every time slab, ${\mathcal W}$ has a constant value with respect to the
time variable, but, with respect to the space variable, is defined as the $L^2(D)-$projection of a
random, piecewise constant function onto the space of linear splines, the computation of which 
leads to the numerical solution of a $(J+1)\times(J+1)$ tridiagonal linear system of
algebraic equations. Finally, ${\mathcal W}$ depends on $N (J+1)$ random variables
and its construction has $O(N\,(J+1))$ complexity, that must to be added to the complexity
of the numerical method used for the approximation of $\uu$. On the contrary, 
the stochastic load ${\mathcal W}$ of the canvas problem \eqref{AC2} we propose here, 
is given explicitly by the formula \eqref{SLoad}, and thus no extra computational cost is required
for its formation. 
\end{remark}
%
%
\subsection{An IMEX finite element method}\label{The_Numerical_Method}
\par
Let $M\in{\mathbb N}$, $\dtau:=\frac{T}{M}$,
$\tau_m:=m\,\dtau$ for $m=0,\dots,M$, and
$\Delta_m:=(\tau_{m-1},\tau_m)$ for $m=1,\dots,M$.
Also, for $r=2$ or $3$,  let ${\sf M}_h^r\subset H^2(D)\cap H_0^1(D)$
be a finite element space consisting
of functions which are piecewise polynomials of degree
at most $r$ over a partition of $D$ in intervals with
maximum mesh-length $h$.
\par
The fully-discrete method we propose for the numerical approximation of $\uu$
uses an implicit/expli\-cit (IMEX) time-discretization treatment of the space
differential operator along with a finite element variational formulation for
space discretization.  Its algorithm is as follows: first sets  
\begin{equation}\label{FullDE1}
\UU_h^0:=0
\end{equation}
and then, for $m=1,\dots,M$, finds $\UU_h^m\in{\sf M}_h^r$
such that
\begin{equation}\label{FullDE2}
\left(\,\UU_h^m-\UU_h^{m-1},\chi\,\right)_{\ssy 0,D}
+\dtau\,\left[\,(\,\partial_x^2\UU_h^{m},\partial_x^2\chi\,)_{\ssy 0,D}
+\mu\,(\,\partial_x^2\UU_h^{m-1},\chi\,)_{\ssy 0,D}\,\right]
=\int_{\ssy\Delta_m}\left(\partial_x{\mathcal
W},\chi\right)_{\ssy 0,D}\,d\tau
\end{equation}
for all $\chi\in{\sf M}_h^r$, where $(\cdot,\cdot)_{\ssy 0,D}$ is the
usual $L^2(D)-$inner product.
\begin{remark}\label{rokoko}
It is easily seen that the numerical method above is  unconditionally stable,
while the Backward Euler finite element method  is stable
under the time-step restriction $\dtau\,\mu^2\leq4$ (see \cite{KZ2013b}).
\end{remark}
%
%
%
%
\subsection{An overview of the paper}
%
%
In Section~\ref{SECTIONCHILD}, we introduce notation and we recall several results that are
often used in the rest of the paper. In Section~\ref{SECTION2}, we focus on the estimation of the error
we made by approximating the solution $u$ to \eqref{PARAB} by the solution $\uu$ to \eqref{AC2},
arriving at the bound
\begin{equation*}
\max_{\ssy[0,T]}\left(\,{\mathbb E}\left[\,
\|u-\uu\|^2_{\ssy L^2(D)}\,\right]\,\right)^{\half}
\leq\,C\,\left(\,{\sf M}^{-\half}+{\Delta t}^{\frac{1}{8}}\,\right)
\end{equation*}
(see Theorem~\ref{BIG_Qewrhma1}).
Section~\ref{Det_Section} is dedicated to the definition and the convergence
analysis of modi\-fied IMEX time-discrete and fully-discrete approximations
of the solution $w$ to the deterministic problem \eqref{Det_Parab}.
The results obtained are used  later in Section~\ref{Stoch_Section}, where we analyze
the numerical method for the approximation of $\uu$, given in Section~\ref{The_Numerical_Method}.
Its convergence is established by proving the following strong error estimate
\begin{equation*}\label{FDEstim4}
\max_{0\leq{m}\leq{\ssy M}}\left(\,
{\mathbb E}\left[\,\|\UU_h^m-\uu(\tau_m,\cdot)\|^2_{\ssy 0,D}\,\right]\right)^{\frac{1}{2}}
\leq\,C\,
\,\left(\,\epsilon^{-\frac{1}{2}}_1\,\dtau^{\frac{1}{8}-\epsilon_1}
+\epsilon_2^{-\frac{1}{2}}\,\,\,h^{\frac{r}{6}-\epsilon_2}\,\right)
\end{equation*}
for all $\epsilon_1\in\left(0,\frac{1}{8}\right]$ and
$\epsilon_2\in\left(0,\frac{r}{6}\right]$ (see Theorem~\ref{FFQEWR}).
We obtain the latter error bound,
by applying a discrete Duhamel principle technique to estimate separately the
{\it time discretization error} and the {\it space discretization error},
which are defined using as an intermediate the corresponding IMEX time-discrete
approximations of $\uu$, specified by \eqref{BackE1} and \eqref{BackE2}
(cf., e.g., \cite{KZ2010}, \cite{KZ2013b}, \cite{YubinY05}).
\par
Since we have no assumptions on the sign, or, the size of $\mu$,
the elliptic operator in \eqref{AC2}
is, in general, not invertible. This is the reason that the
Backward Euler/finite element method is stable and convergent
after adopting a restriction on the time-step size (see \cite{KZ2013b},
Remark~\ref{rokoko}).
On the contrary, the IMEX/finite element method we propose here
is unconditionally stable and convergent, because
the principal part of the elliptic operator is treated implicitly and
its lower order part explicitly.
Another characteristic in our method is the choice to build up the canvas
problem using spectral functions, which allow us to avoid the numerical
solution of an extra linear system of algebraic equation at every time step
that is required in the approach of \cite{KZ2013b}
(see Remark~\ref{complex_rem}).
\par
The error analysis of the IMEX finite element method is more technical than that in \cite{KZ2013b}
for the Backward Euler finite element method. The main difference is due to the fact that
the representation of the time-discrete and fully discrete approximations of $\uu$ is related
to a modified version of the IMEX time-stepping method for the approximation of
the solution to the deterministic problem \eqref{Det_Parab}, the error analysis of which is
necessary in obtaining the desired error estimate and is of independent interest
(see Section~\ref{Det_Section}).
%
%
%
%
%
%
%
%
%
%
%
%
%
%
%
%
%
\section{Preliminaries}\label{SECTIONCHILD}
We denote by $L^2(D)$ the space of the Lebesgue measurable
functions which are square integrable on $D$ with respect to
the Lebesgue measure $dx$. The space $L^2(D)$ is provided
with the standard norm $\|g\|_{\ssy 0,D}:= \left(\int_{\ssy D}|g(x)|^2\,dx\right)^{\frac{1}{2}}$
for $g\in L^2(D)$, which is derived by the usual inner product
 $(g_1,g_2)_{\ssy 0,D}:=\int_{\ssy D}g_1(x)\,g_2(x)\,dx$
for $g_1$, $g_2\in L^2(D)$.
Also, we employ the symbol ${\mathbb N}_0$ for the set of all nonnegative integers.
\par
For $s\in\nset_0$, we denote by $H^s(D)$ the Sobolev space of functions
having generalized derivatives up to order $s$ in
$L^2(D)$, and by $\|\cdot\|_{\ssy s,D}$ its usual norm, i.e.
$\|g\|_{\ssy s,D}:=\left(\sum_{\ell=0}^s
\|\partial_x^{\ell}g\|_{\ssy 0,D}^2\right)^{\ssy 1/2}$
for $g\in H^s(D)$. Also, by $H_0^1(I)$ we denote the subspace of $H^1(D)$
consisting of functions which vanish at the endpoints of $D$ in
the sense of trace.
%
%
%
\par
The sequence of pairs
$\left\{\left(\lambda_i^2,\varepsilon_i\right)\right\}_{i=1}^{\infty}$ is a
solution to the eigenvalue/eigenfunction problem: find nonzero
$\varphi\in H^2(D)\cap H_0^1(D)$ and $\lambda\in\rset$ such that
$-\varphi''=\lambda\,\varphi$ in $D$. Since
$(\varepsilon_i)_{i=1}^{\infty}$ is a complete
$(\cdot,\cdot)_{\ssy 0,D}-$orthonormal system in $L^2(D)$, for
$s\in\rset$, we define by
\begin{equation*}
{\mathcal V}^s(D):=\Big\{v\in L^2(D):\quad\sum_{i=1}^{\infty}
\lambda_{i}^{2s} \,(v,\varepsilon_i)^2_{\ssy 0,D}<\infty\,\Big\}
\end{equation*}
a subspace of $L^2(D)$ provided with the natural norm
$\|v\|_{\ssy{\mathcal V}^s}:=\big(\,\sum_{i=1}^{\infty}
\lambda_{i}^{2s}\,(v,\varepsilon_i)^2_{\ssy
0,D}\,\big)^{\ssy 1/2}$ for $v\in{\mathcal V}^s(D)$.
For $s\ge 0$, the space $({\mathcal V}^s(D),\|\cdot\|_{\ssy{\mathcal V}^s})$ is a complete
subspace of $L^2(D)$ and we define
$({\bfdot H}^s(D),\|\cdot\|_{\ssy{\bfdot H}^s})
:=({\mathcal V}^s(D),\|\cdot\|_{\ssy{\mathcal V}^s})$.
For $s<0$, the space $({\bfdot H}^s(D),\|\cdot\|_{\ssy{\bfdot H}^s})$  is defined
as the completion of $({\mathcal V}^s(D),\|\cdot\|_{\ssy{\mathcal V}^s})$, or, equivalently, as the dual of
 $({\bfdot H}^{-s}(D),\|\cdot\|_{\ssy{\bfdot H}^{-s}})$.
\par
Let $m\in\nset_0$. It is well-known (see \cite{Thomee}) that
\begin{equation*}\label{dot_charact}
{\bfdot H}^m(D)=\left\{\,v\in H^m(D):
\quad\partial^{2\ell}v\left|_{\ssy\partial D}\right.=0
\quad\text{\rm if}\ken 0\leq{2\ell}<m\,\right\}
\end{equation*}
and that there exist constants $C_{m,{\ssy A}}$ and $C_{m,{\ssy B}}$
such that
\begin{equation}\label{H_equiv}
C_{m,{\ssy A}}\,\|v\|_{\ssy m,D} \leq\|v\|_{\ssy{\bfdot H}^m}
\leq\,C_{m,{\ssy B}}\,\|v\|_{\ssy m,D}\quad \forall\,v\in{\bfdot
H}^m(D).
\end{equation}
Also, we define on $L^2(D)$ the negative norm $\|\cdot\|_{\ssy -m,
D}$ by
\begin{equation*}
\|v\|_{\ssy -m, D}:=\sup\Big\{ \tfrac{(v,\varphi)_{\ssy 0,D}}
{\|\varphi\|_{\ssy m,D}}:\quad \varphi\in{\bfdot H}^m(D)
\ken\text{\rm and}\ken\varphi\not=0\Big\} \quad\forall\,v\in
L^2(D),
\end{equation*}
for which, using \eqref{H_equiv}, follows that
there exists a constant $C_{-m}>0$ such that
\begin{equation}\label{minus_equiv}
\|v\|_{\ssy -m,D}\leq\,C_{-m}\,\|v\|_{{\bfdot H}^{-m}}
\quad\forall\,v\in L^2(D).
\end{equation}
\par
Let ${\mathbb L}_2=(L^2(D),(\cdot,\cdot)_{\ssy 0,D})$ and
${\mathcal L}({\mathbb L}_2)$ be the space of linear, bounded
operators from ${\mathbb L}_2$ to ${\mathbb L}_2$. An
operator $\Gamma\in {\mathcal L}({\mathbb L}_2)$ is 
Hilbert-Schmidt, when $\|\Gamma\|_{\ssy\rm
HS}:=\left(\sum_{i=1}^{\infty} \|\Gamma\varepsilon_i\|^2_{\ssy
0,D}\right)^{\half}<+\infty$, where $\|\Gamma\|_{\ssy\rm HS}$ is
the so called Hilbert-Schmidt norm of $\Gamma$.
%
%
We note that the quantity $\|\Gamma\|_{\ssy\rm HS}$ does not
change when we replace $(\varepsilon_i)_{i=1}^{\infty}$ by
another complete orthonormal system of ${\mathbb L}_2$.
It is well known (see, e.g., \cite{DunSch}, \cite{LPS}) that an operator
$\Gamma\in{\mathcal L}({\mathbb L}_2)$ is Hilbert-Schmidt iff
there exists a measurable function $\gamma:D\times D\rightarrow{\mathbb
R}$ such that $\Gamma[v](\cdot)=\int_{\ssy D}\gamma(\cdot,y)\,v(y)\,dy$
for $v\in L^2(D)$, and then, it holds that
%
%
\begin{equation}\label{HSxar}
\|\Gamma\|_{\ssy\rm HS} =\left(\,
\iint_{\ssy D\times D}\gamma^2(x,y)\,dxdy\,\right)^{\half}.
\end{equation}
Let ${\mathcal L}_{\ssy\rm HS}({\mathbb L}_2)$ be the set of
Hilbert Schmidt operators of ${\mathcal L}({\mathbb L}^2)$ and
$\Phi:[0,T]\rightarrow {\mathcal L}_{\ssy\rm HS}({\mathbb L}_2)$.
Also, for a random variable $X$, let ${\mathbb E}[X]$ be its
expected value, i.e., ${\mathbb E}[X]:=\int_{\ssy\Omega}X\,dP$.
Then, the It{\^o} isometry property for stochastic integrals  reads
\begin{equation}\label{Ito_Isom}
{\mathbb E}\left[\Big\|\int_0^{\ssy T}\Phi\,dW\Big\|_{\ssy 0,D}^2\right]
=\int_0^{\ssy T}\|\Phi(t)\|_{\ssy\rm HS}^2\,dt.
\end{equation}
\par
%
%
For later use, we recall that if $({\mathcal H},(\cdot,\cdot)_{\ssy{\mathcal H}})$ is a
real inner product space with induced norm $|\cdot|_{\ssy{\mathcal H}}$,
then
\begin{equation}\label{innerproduct}
2\,(g-v,g)_{\ssy{\mathcal H}}
=|g|^2_{\ssy{\mathcal H}}-|v|^2_{\ssy{\mathcal H}}
+|g-v|^2_{\ssy{\mathcal H}}\quad\forall\,g,v\in{\mathcal H}.
\end{equation}
\par
Finally, for any nonempty set $A$, we denote by ${\mathcal X}_{\ssy A}$
the indicator function of $A$.
%
%
%
%
%

%
%
%
%
%
%
\subsection{A projection operator}
Let ${\mathcal O}:=(0,T)\times D$,
${\mathfrak S}_{\ssy{\sf M}}:=\mathop{\rm span}(\varphi_i)_{i=1}^{\ssy{\sf M}}$,
${\mathfrak S}_{\ssy{\sf N}}:=\mathop{\rm span}({\mathcal X}_{\ssy T_n})_{n=1}^{\ssy{\sf N}}$
%
%
and ${\sf\Pi}:L^2({\mathcal O})\rightarrow{\mathfrak S}_{\ssy{\sf N}}\otimes{\mathfrak S}_{\ssy{\sf M}}$
the usual $L^2({\mathcal O})-$projection operator which is given by the formula
\begin{equation}\label{Defin_L2}
{\sf\Pi}g:=\tfrac{1}{\Delta{t}}\,\sum_{i=1}^{\ssy{\sf M}}\,\left(\,\sum_{n=1}^{\ssy{\sf N}}
{\mathcal X}_{\ssy{T_n}}\,\int_{\ssy T_n}
(g,\varphi_i)_{\ssy 0,D}\,dt\right)\,\varphi_i
\quad\forall\,g\in L^2({\mathcal O}).
\end{equation}
%
%
Then, the following representation of the stochastic integral of ${\sf\Pi}$ holds (cf. Lemma~2.1 in \cite{KZ2010}).
%
%
%
\begin{lemma}\label{Lhmma1}
For $g\in L^2({\mathcal O})$, it holds that
\begin{equation}\label{WNEQ2}
\int_0^{\ssy T}\!\!\int_{\ssy D}{\sf\Pi}g(t,x)\,dW(t,x)
=\iint_{\ssy{\mathcal O}}{\mathcal W}(s,y)\,g(s,y)\,dsdy.
\end{equation}
\end{lemma}
%
%
%
%
%
%
%
%
%
%
%
%
%
%
%
%
%
%
\begin{proof}
Using \eqref{Defin_L2} and \eqref{RDef}, we have
\begin{equation*}
\begin{split}
\int_0^{\ssy T}\!\!\int_{\ssy D}{\sf\Pi}g(t,x)\,dW(t,x)
=&\tfrac{1}{\Delta{t}}\, \sum_{n=1}^{\ssy{\sf N}}
\,\sum_{i=1}^{\ssy{\sf M}}\left(\,
\int_{\ssy T_n}\!\int_{\ssy D}g(s,y)\,\varphi_i(y)\;dsdy\,\right)\, R_i^{n}\\
=&\tfrac{1}{\Delta{t}}\, \sum_{n=1}^{\ssy{\sf N}}
\,\sum_{i=1}^{\ssy{\sf M}}\left(\,
\iint_{\ssy{\mathcal O}}{\mathcal X}_{\ssy T_n}(s)\,R_i^{n}\,g(s,y)\,\varphi_i(y)\;dsdy
\,\right)\\
=&\,\iint_{\ssy{\mathcal O}}g(s,y)\,\left (\,\tfrac{1}{\Delta{t}}\,\sum_{n=1}^{\ssy{\sf N}}
\,\sum_{i=1}^{\ssy{\sf M}}{\mathcal X}_{\ssy T_n}(s)\,R_i^{n}\,\varphi_i(y)\,\right)\;dsdy\\
\end{split}
\end{equation*}
which along \eqref{SLoad} yields \eqref{WNEQ2}.
\end{proof}
%
%
%
%
\subsection{Linear elliptic and parabolic operators}\label{SECTION31}
%
Let $T_{\ssy E}:L^2(D)\rightarrow{\bfdot H}^2(D)$ be the solution
operator of the Dirichlet two-point boundary
value problem: for given $f\in L^2(D)$ find $v_{\ssy E}\in {\bfdot
H}^2(D)$ such that $v_{\ssy E}''=f$ in $D$,
i.e. $T_{\ssy E}f:=v_{\ssy E}$. It is well-known that 
\begin{equation}\label{TE_Prop}
(T_{\ssy E}f,g)_{\ssy 0,D}=(f,T_{\ssy E}g)_{\ssy 0,D}
\quad\forall\,f,g\in L^2(D)
\end{equation}
and, for $m\in{\mathbb N}_0$, there exists a constant $C_{\ssy E}^m>0$ such that
\begin{equation}\label{ElReg1}
\|T_{\ssy E}f\|_{\ssy m,D}\leq \,C_{\ssy E}^m \,\|f\|_{\ssy m-2, D}
\quad\forall\,f\in H^{\max\{0,m-2\}}(D).
\end{equation}
\par
Let, also, $T_{\ssy B}:L^2(D)\rightarrow{\bfdot H}^4(D)$
be the solution operator of the following Dirichlet biharmonic two-point
boundary value problem: for given $f\in L^2(D)$ find $v_{\ssy B}\in
{\bfdot H}^4(D)$ such that
\begin{equation}\label{ElOp2}
v_{\ssy B}''''=f\quad\text{\rm in}\ken D,
\end{equation}
i.e. $T_{\ssy B}f:=v_{\ssy B}$. It is well-known that, for $m\in{\mathbb N}_0$, there
exists a constant $C^{m}_{\ssy B}>0$ such that
\begin{equation}\label{ElBihar2}
\|T_{\ssy B}f\|_{\ssy m,D}\leq \,C^m_{\ssy B}\,\|f\|_{\ssy m-4, D}
\quad\forall\,f\in H^{\max\{0,m-4\}}(D).
\end{equation}
Due to the type of boundary conditions of \eqref{ElOp2}, we have
\begin{equation}\label{tiger007}
T_{\ssy B}f= T_{\ssy E}^2f\quad\forall\,f\in L^2(D),
\end{equation}
which, after using \eqref{TE_Prop}, yields
\begin{equation}\label{TB-prop1}
(T_{\ssy B}v_1,v_2)_{\ssy 0,D} =(T_{\ssy E}v_1,T_{\ssy
E}v_2)_{\ssy 0,D}=(v_1,T_{\ssy B}v_2)_{\ssy 0,D}
\quad\forall\,v_1,v_2\in L^2(D).
\end{equation}
%
%
%
%
%
%
%
\par
Let $({\mathcal S}(t)w_0)_{\ssy t\in[0,T]}$ be the standard
semigroup notation for the solution $w$ to \eqref{Det_Parab}.
Then (see Appendix~A in \cite{KZ2013b}) for $\ell\in{\mathbb N}_0$,
$\beta\ge0$ and $p\ge0$, there exists a constant
${\mathcal C}_{\beta,\ell,\mu,\mu^2 T}>0$ such that
\begin{equation}\label{Reggo3}
\int_{t_a}^{t_b}(\tau-t_a)^{\beta}\,
\big\|\partial_t^{\ell}{\mathcal S}(\tau)w_0 \big\|_{\ssy {\bfdot
H}^p}^2\,d\tau \leq\,{\mathcal C}_{\beta,\ell,\mu,\mu^2 T}\, \|w_0\|^2_{\ssy {\bfdot
H}^{p+4\ell-2\beta-2}}
\end{equation}
for all $w_0\in{\bfdot H}^{p+4\ell-2\beta-2}(D)$ and $t_a$, $t_b\in[0,T]$ with $t_b>t_a$.
%
%
%
%
\subsection{Discrete operators}
%
%
Let $r=2$ or $3$,  and ${\sf M}_h^r\subset H_0^1(D)\cap H^2(D)$
be a finite element space consisting of functions which are
piecewise polynomials of degree at most $r$ over a partition of
$D$ in intervals with maximum length $h$. It is well-known (cf.,
e.g., \cite{BrHilbert1970}) that
\begin{equation}\label{Sh_H2}
\inf_{\chi\in{\sf M}_h^r} \|v-\chi\|_{\ssy 2,D} \leq\,C_{r}\,h^{s-2}\,\|v\|_{\ssy s,D}
\quad\,\forall\,v\in H^{s+1}(D)\cap H_0^1(D),\quad\,s=3,\dots,r+1,
\end{equation}
where $C_{r}$ is a positive constant that depends on $r$ and $D$, and is independent of $h$ and $v$.
Then, we define the discrete biharmonic operator $B_h:{\sf M}_h^r\to{\sf M}_h^r$ by
$(B_h\varphi,\chi)_{\ssy 0,D}=(\partial_x^2\varphi,\partial_x^2\chi)_{\ssy 0,D}$
for $\varphi,\chi\in{\sf M}_h^r$, the $L^2(D)-$projection operator
$P_h:L^2(D)\to{\sf M}_h^r$  by
$(P_hf,\chi)_{\ssy 0,D}=(f,\chi)_{\ssy 0,D}$ for $\chi\in{\sf M}_h^r $ and 
$f\in L^2(D)$,  and the standard Galerkin finite element approximation
$v_{{\ssy B},h}\in{\sf M}_h^r$ of the solution
$v_{\ssy B}$ to \eqref{ElOp2} by requiring
\begin{equation}\label{April2017_1}
B_h(v_{{\ssy B},h})=P_hf.
\end{equation}
\par
Let $T_{{\ssy B},h}:L^2(D)\to{{\sf M}^r_h}$ be the solution operator of the 
finite element method \eqref{April2017_1}, i.e. $T_{{\ssy B},h}f:=v_{{\ssy B},h}=B_h^{-1}P_hf$
for all $f\in L^2(D)$. Then, we can easily conclude that
\begin{equation}\label{adjo2}
\left(T_{{\ssy B},h}f,g\right)_{\ssy 0,D}=\left(\partial_x^2\left(T_{{\ssy B},h}f\right),
\partial_x^2\left(T_{{\ssy B},h}g\right)\right)_{\ssy 0,D}
=\left(f,T_{{\ssy B},h}g\right)_{\ssy 0,D} \quad\forall\,f,g\in L^2(D)
\end{equation}
and
\begin{equation}\label{TB_bound}
\|\partial_x^2(T_{\ssy B,h}f)\|_{\ssy 0,D}
\leq\,C\,\|f\|_{\ssy -2,D}\quad\forall\,f\in L^2(D).
\end{equation}
\par
Finally, the approximation property \eqref{Sh_H2} of the finite element space
${\sf M}_h^r$ yields (see, e.g., Proposition 2.2 in \cite{KZ2010}) the following
error estimate
\begin{equation}\label{ARA1}
\|T_{\ssy B}f-T_{{\ssy B},h}f\|_{\ssy 0,D} \leq\,C\,
h^{r}\,\|f\|_{\ssy -1,D}\quad\forall\,f\in L^2(D),\quad r=2,3.
\end{equation}
%
%
%
%
%
%
%
\section{An approximation estimate for the canvas problem solution}\label{SECTION2}
Here, we establish the convergence of $\uu$ towards $u$ with respect to the
$L^{\infty}_t(L^2_{\ssy P}(L^2_x))$ norm,
when $\Delta{t}\rightarrow0$ and ${\sf M}\rightarrow\infty$  (cf. \cite{KZ2010}, \cite{KZ2013b}).
%
%
\begin{theorem}\label{BIG_Qewrhma1}
Let $u$ be the solution to \eqref{PARAB}, $\uu$ be the
solution to \eqref{AC2} and $\kappa\in{\mathbb N}$ such that $\kappa^2\,\pi^2>\mu$.
Then, there exists a constant ${\widehat c}_{\ssy{\rm CER}}>0$,
independent of $\Delta{t}$ and ${\sf M}$, such that
\begin{equation}\label{ModelError}
\max_{\ssy[0,T]}{\sf\Theta}
\leq\,{\widehat c}_{\ssy{\rm CER}}\,\left(\,\Delta{t}^\frac{1}{8}
+{\sf M}^{-\frac{1}{2}}\,\right)\quad\forall\,{\sf M}\ge\kappa,
\end{equation}
where ${\sf\Theta}(t):=\left({\mathbb E}\left[\|u(t,\cdot)-\uu(t,\cdot)\|_{\ssy 0,D}^2
\right]\right)^{\half}$ for $t\in[0,T]$.
\end{theorem}
%
%
%
%
%
\begin{proof} In the sequel,  we will use the symbol
$C$ to denote a generic constant that is independent of $\Delta{t}$
and ${\sf M}$ and may change value from one line to the other.
\par
Using \eqref{MildSol}, \eqref{HatUform} and Lemma~\ref{Lhmma1}, we
conclude that
\begin{equation}\label{corv00}
u(t,x)-\uu(t,x)=\int_0^{\ssy T}\!\!\!\int_{\ssy D}
\big[{\mathcal X}_{(0,t)}(s)\,{\sf\Psi}_{t-s}(x,y) -{\widetilde
{\sf\Psi}}(t,x;s,y)\big]\,dW(s,y)
\end{equation}
for $(t,x)\in[0,T]\times{\overline D}$,
where ${\widetilde{\sf\Psi}}: (0,T)\times{D}\rightarrow
L^2({\mathcal D})$ is given by
\begin{equation*}
{\widetilde{\sf\Psi}}(t,x;s,y):=\tfrac{1}{\Delta{t}}\sum_{i=1}^{\ssy{\sf M}}
\left[\,\int_{\ssy T_n}{\mathcal X}_{(0,t)}(s')
\left(\int_{\ssy D}{\sf\Psi}_{t-s'}(x,y')\,\varphi_i(y')\,dy'\right)ds'\right]\varphi_i(y)
\end{equation*}
for $(s,y)\in T_n\times{D}$, $n=1,\dots,{\sf N}$, and for
$(t,x)\in(0,T]\times D$. Now, we use \eqref{GreenKernel} and the $L^2(D)-$orthogonality
of $(\varphi_k)_{k=1}^{\infty}$  to obtain
\begin{equation}\label{X_Nov_2015}
{\widetilde{\sf\Psi}}(t,x;s,y)
=\tfrac{1}{\Delta{t}}\,
\int_{\ssy T_n}{\mathcal X}_{(0,t)}(s')\left(\,
\sum_{i=1}^{\ssy{\sf M}}\lambda_i e^{-\lambda_i^2(\lambda_i^2-\mu)(t-s')}
\varepsilon_i(x)\,\varphi_i(y)\right)\,ds'
\end{equation}
for $(s,y)\in T_n\times{D}$, $n=1,\dots,{\sf N}$, and for
$(t,x)\in(0,T]\times D$. Also, we use \eqref{corv00}, \eqref{Ito_Isom}
and \eqref{HSxar}, to get
\begin{equation}\label{corv1}
\begin{split}
{\sf\Theta}(t)=&\,\left(\,\int_0^{\ssy T}\int_{\ssy D}\int_{\ssy D}
 \big[{\mathcal X}_{(0,t)}(s)\,{\sf\Psi}_{t-s}(x,y)
-{\widetilde{\sf\Psi}}(t,x;s,y)\big]^2\;dxdyds\,\right)^{\frac{1}{2}}\\
\leq&\,\sqrt{{\sf\Theta}_{\ssy A}(t)}+\sqrt{{\sf\Theta}_{\ssy B}(t)}\quad\forall\,t\in(0,T],\\
\end{split}
\end{equation}
where
\begin{equation*}\label{ZaZa}
{\sf\Theta}_{\ssy A}(t):=
\sum_{n=1}^{\ssy{\sf N}}\int_{\ssy D}\!\int_{\ssy D}\!\int_{\ssy T_n}
\left[{\mathcal X}_{(0,t)}(s)\,{\sf\Psi}_{t-s}(x,y)
-\tfrac{1}{\Delta{t}}\,\int_{\ssy T_n}
{\mathcal X}_{(0,t)}(s')\,{\sf\Psi}_{t-s'}(x,y)\,ds'\,\right]^2\;dxdyds
\end{equation*}
and
\begin{equation*}\label{ZbZb}
{\sf\Theta}_{\ssy B}(t):=\sum_{n=1}^{\ssy{\sf N}}
\int_{\ssy D}\int_{\ssy D}\int_{\ssy T_n}\left[\,
\tfrac{1}{\Delta{t}}\int_{\ssy T_n}{\mathcal X}_{(0,t)}(s')\,{\sf\Psi}_{t-s'}(x,y)\,ds'
-{\widetilde{\sf\Psi}}(t,x;s,y)\,\right]^2dxdyds.
\end{equation*}
Proceeding as in the proof of Theorem~3.1 in \cite{KZ2013b} we arrive at
\begin{equation}\label{ZaZa_bound}
\sqrt{{\sf\Theta}_{\ssy A}(t)}\leq\,C\,\Delta{t}^{\frac{1}{8}}\quad\forall\,t\in(0,T].
\end{equation}
Combining \eqref{ZbZb} and \eqref{X_Nov_2015}
and using the $L^2(D)-$orthogonality of $(\varepsilon_k)_{k=1}^{\infty}$
and  $(\varphi_k)_{k=1}^{\infty}$ we have
\begin{equation*}
\begin{split}
{\sf\Theta}_{\ssy B}(t)=&\,\tfrac{1}{\Delta{t}}\,
\sum_{n=1}^{\ssy{\sf N}}
\int_{\ssy D}\int_{\ssy D}\left[
\int_{\ssy T_n}{\mathcal X}_{(0,t)}(s')
\left(\,{\sf\Psi}_{t-s'}(x,y)-\sum_{i=1}^{\ssy{\sf M}}
\lambda_ie^{-\lambda_i^2(\lambda_i^2-\mu)(t-s')}
\varepsilon_i(x)\,\varphi_i(y)\right)ds'\right]^2dxdy\\
=&\,\tfrac{1}{\Delta{t}}\,
\sum_{n=1}^{\ssy{\sf N}}
\int_{\ssy D}\int_{\ssy D}\left[
\int_{\ssy T_n}{\mathcal X}_{(0,t)}(s')
\left(\,\sum_{i={\ssy{\sf M}}+1}^{\infty}
\lambda_ie^{-\lambda_i^2(\lambda_i^2-\mu)(t-s')}
\varepsilon_i(x)\,\varphi_i(y)\right)ds'
\right]^2dxdy\\
=&\,\tfrac{1}{\Delta{t}}\,
\sum_{n=1}^{\ssy{\sf N}}
\int_{\ssy D}\int_{\ssy D}\left[
\,\sum_{i={\ssy{\sf M}}+1}^{\infty}
\left( \int_{\ssy T_n}{\mathcal X}_{(0,t)}(s')
\,\lambda_ie^{-\lambda_i^2(\lambda_i^2-\mu)(t-s')}\;ds' \right)
\varepsilon_i(x)\,\varphi_i(y)
\right]^2dxdy\\
=&\,\tfrac{1}{\Delta{t}}\,
\sum_{n=1}^{\ssy{\sf N}}
\,\sum_{i={\ssy{\sf M}}+1}^{\infty}
\left( \int_{\ssy T_n}{\mathcal X}_{(0,t)}(s')\,\lambda_i\,e^{-\lambda_i^2(\lambda_i^2-\mu)(t-s')}\;ds' \right)^2
\quad\forall\,t\in(0,T].
\end{split}
\end{equation*}
For ${\sf M}\ge\kappa$, using the Cauchy-Schwarz inequality, we obtain
\begin{equation}\label{ZbZb_bound}
\begin{split}
\sqrt{{\sf\Theta}_{\ssy B}(t)}\leq&\,\left[ \,\sum_{i={\ssy{\sf M}}+1}^{\infty}
\lambda_i^2\,\left(\,\int_0^t\,e^{-2\,\lambda_i^2(\lambda_i^2-\mu)(t-s')}\;ds'\,\right)
\right]^{\frac{1}{2}}\\
\leq&\,\tfrac{1}{\sqrt{2}}\,\left(\,\sum_{i={\ssy{\sf M}}+1}^{\infty}
\tfrac{1}{\lambda_i^2-\mu}\,\right)^{\frac{1}{2}}\\
\leq&\,\tfrac{\kappa+1}{\sqrt{2+4\kappa}}
\,\left(\,\sum_{i={\ssy{\sf M}}+1}^{\infty}
\tfrac{1}{\lambda_i^2}\,\right)^{\frac{1}{2}}\\
\leq&\,\tfrac{\kappa+1}{\pi\,\sqrt{2+4\kappa}}
\,\left(\,\int_{\ssy{\sf M}}^{\infty}\tfrac{1}{x^2}\;dx\,\right)^{\frac{1}{2}}\\
\leq&\,\tfrac{\kappa+1}{\pi\,\sqrt{2+4\kappa}}\,{\sf M}^{-\half}\quad\forall\,t\in(0,T].\\
\end{split}
\end{equation}
\par
The error bound \eqref{ModelError} follows by observing that
$\Theta(0)=0$ and by combining the bounds \eqref{corv1},
\eqref{ZaZa_bound} and \eqref{ZbZb_bound}.
\end{proof}
%
%
%
%
\section{Deterministic Time-Discrete and Fully-Discrete Approximations}\label{Det_Section}
In this section we define and analyze auxiliary time-discrete and fully-discrete 
approximations  of the solution to the deterministic problem \eqref{Det_Parab}.
The results of the
convergence analysis will be used in Section~\ref{Stoch_Section}
for the derivation of  an error estimate for numerical approximations
of  $\uu$ introduced in Section~\ref{The_Numerical_Method}.
%
%
%
\subsection{Time-Discrete Approximations}\label{Det_TD}
We define an auxiliary  modified-IMEX time-discrete method
to approximate the solution $w$ to \eqref{Det_Parab}, which has the 
following structure: First sets
\begin{equation}\label{BEDet1}
W^0:=w_0
\end{equation}
and determines $W^1\in{\bfdot H}^4(D)$ by
\begin{equation}\label{BEDet2a}
W^1-W^{0} +\Delta{\tau}\,\partial_x^4W^1=0.
\end{equation}
Then, for $m=2,\dots,M$,  finds $W^m\in{\bfdot H}^4(D)$
such that
\begin{equation}\label{BEDet2}
W^m-W^{m-1}+\dtau\,\left(\,\partial_x^4W^m
+\mu\,\partial_x^2W^{m-1}\,\right)=0.
\end{equation}
\par
In the proposition below, we derive a low regularity priori error estimate
in a discrete in time $L^2_t(L^2_x)-$norm.
%
%
%
\begin{proposition}\label{DetPropo1}
Let $(W^m)_{m=0}^{\ssy M}$ be the time-discrete approximations
defined in \eqref{BEDet1}--\eqref{BEDet2},
and $w$ be the solution to the problem \eqref{Det_Parab}.
Then, there exists a constant $C>0$, independent of $\dtau$, such that
\begin{equation}\label{Nov2015_Propo1}
\left(\,\dtau\,\sum_{m=1}^{\ssy M}
\|W^m-w^m\|_{\ssy 0,D}^2 \,\right)^{\frac{1}{2}}
\leq\,C\,\dtau^{\theta} \,\|w_0\|_{\ssy{\bfdot H}^{4\theta-2}}
\quad\forall\,\theta\in[0,1],
\quad\forall\,w_0\in{\bfdot H}^2(D),
\end{equation}
where $w^{\ell}(\cdot):=w(\tau_{\ell},\cdot)$ for $\ell=0,\dots,M$.
\end{proposition}
%
%
%
%
%
%
%
%
\begin{proof}
In the sequel,  we will use the symbol $C$ to denote a generic constant that is
independent of $\dtau$ and may changes value from one line to the other.
\par
Let ${\sf E}^m:=w^m-W^m$ for $m=0,\dots,M$, and
\begin{equation*}
\sigma_{m}(\cdot):=\int_{\ssy\Delta_m}
\left(\,w(\tau_m,\cdot)-w(\tau,\cdot)\,\right)\,d\tau\,
+\mu\,\int_{\ssy\Delta_m}
T_{\ssy E}\left(\,w(\tau_{m-1},\cdot)-w(\tau,\cdot)\,\right)\,d\tau
\end{equation*}
for $m=1,\dots,M$.
Thus, combining \eqref{Det_Parab},  \eqref{BEDet2a} and \eqref{BEDet2},
we conclude that
\begin{equation}\label{Nov2015_1}
T_{\ssy B}({\sf E}^1-{\sf E}^0)+\dtau\,{\sf E}^1=\sigma_1-\dtau\,\mu\,T_{\ssy E}w_0,
\end{equation}
\begin{equation}\label{Nov2015_2}
T_{\ssy B}({\sf E}^m-{\sf E}^{m-1})+\dtau\,\left(\,{\sf E}^m+\mu\,T_{\ssy
E}{\sf E}^{m-1}\,\right) =\sigma_m,\quad m=2,\dots,M.
\end{equation}
\par
First take the $L^2(D)-$inner product of both sides of
\eqref{Nov2015_1} with ${\sf E}^1$ and of 
\eqref{Nov2015_2} with ${\sf E}^m$, and then
use \eqref{TB-prop1} to obtain
\begin{equation*}
(T_{\ssy E}{\sf E}^1-T_{\ssy E}{\sf E}^0,T_{\ssy E}{\sf E}^1)_{\ssy 0,D}
+\dtau\,\|{\sf E}^1\|_{\ssy 0,D}^2=(\sigma_1,{\sf E}^1)_{\ssy 0,D}
-\dtau\,\mu\,(T_{\ssy E}w_0,{\sf E}^1)_{\ssy 0,D},
\end{equation*}
\begin{equation*}
(T_{\ssy E}{\sf E}^m-T_{\ssy E}{\sf E}^{m-1},T_{\ssy E}{\sf E}^m)_{\ssy 0,D}
+\dtau\,\|{\sf E}^m\|_{\ssy 0,D}^2=
-\mu\,\dtau\,(T_{\ssy E}{\sf E}^{m-1},{\sf E}^m)_{\ssy 0,D}
+(\sigma_m,{\sf E}^m)_{\ssy 0,D}
\end{equation*}
for $m=2,\dots,M$. Then, using that ${\sf E}^0=0$ and applying
\eqref{innerproduct} along with the arithmetic mean inequality, we get
\begin{equation}\label{Nov2015_3}
\|T_{\ssy E}{\sf E}^1\|_{\ssy 0,D}^2
+\dtau\,\|{\sf E}^1\|_{\ssy 0,D}^2
\leq\dtau^{-1}\,\|\sigma_1\|_{\ssy 0,D}^2
-2\,\dtau\,\mu\,(T_{\ssy E}w_0,{\sf E}^1)_{\ssy 0,D},
\end{equation}
\begin{equation}\label{Nov2015_4}
\|T_{\ssy E}{\sf E}^m\|_{\ssy 0,D}^2
+\tfrac{1}{2}\,\dtau\,\|{\sf E}^m\|_{\ssy 0,D}^2\leq
(1+2\,\mu^2\,\dtau)\,\|T_{\ssy E}{\sf E}^{m-1}\|_{\ssy 0,D}^2
+\dtau^{-1}\,\|\sigma_m\|_{\ssy 0,D}^2,\quad m=2,\dots,M.
\end{equation}
Observing that \eqref{Nov2015_4} yields
\begin{equation*}\label{Nov2015_7}
\|T_{\ssy E}{\sf E}^m\|_{\ssy 0,D}^2\leq\,(1+2\,\mu^2\,\dtau)
\,\|T_{\ssy E}{\sf E}^{m-1}\|_{\ssy 0,D}^2
+\dtau^{-1}\,\|\sigma_m\|_{\ssy 0,D}^2,\quad
m=2,\dots,M,
\end{equation*}
we use a standard discrete Gronwall argument to arrive at
\begin{equation}\label{Nov2015_8}
\max_{1\leq{m}\leq{\ssy M}}\|T_{\ssy E}{\sf E}^m\|_{\ssy 0,D}^2
\leq\,C\,\left(\,\|T_{\ssy E}{\sf E}^1\|_{\ssy 0,D}^2
+\dtau^{-1}\,\sum_{m=2}^{\ssy M}\,\|\sigma_m\|_{\ssy 0,D}^2\,\right).
\end{equation}
Summing both sides of \eqref{Nov2015_4} with respect to $m$,
from 2 up to $M$,  we obtain
\begin{equation*}
\|T_{\ssy E}{\sf E}^{\ssy M}\|_{\ssy 0,D}^2
+\tfrac{\dtau}{2}\,\sum_{m=2}^{\ssy M}\|{\sf E}^m\|_{\ssy
0,D}^2\leq\|T_{\ssy E}{\sf E}^1\|_{\ssy 0,D}^2
+2\,\mu^2\,\dtau\,\sum_{m=1}^{\ssy M-1}\|T_{\ssy E}{\sf E}^{m}\|_{\ssy 0,D}^2
+\dtau^{-1}\sum_{m=2}^{\ssy M}\|\sigma_m\|_{\ssy 0,D}^2,
\end{equation*}
which, along with \eqref{Nov2015_8}, yields
\begin{equation}\label{Nov2015_9}
\dtau\,\sum_{m=1}^{\ssy M}\|{\sf E}^m\|_{\ssy 0,D}^2\leq\,C\,\left(\,
\|T_{\ssy E}{\sf E}^1\|_{\ssy 0,D}^2+\dtau\,\|{\sf E}^1\|^2_{\ssy 0,D}
+\dtau^{-1}\,\sum_{m=2}^{\ssy M}\|\sigma_m\|_{\ssy 0,D}^2\,\right).
\end{equation}
Using \eqref{Nov2015_3}, \eqref{TE_Prop},
the Cauchy-Schwarz inequality and the arithmetic mean inequality, we have
\begin{equation*}
\begin{split}
\|T_{\ssy E}{\sf E}^1\|_{\ssy 0,D}^2
+\dtau\,\|{\sf E}^1\|_{\ssy 0,D}^2\leq&\,\dtau^{-1}\,\|\sigma_1\|_{\ssy 0,D}^2
-2\,\dtau\,\mu\,(w_0,T_{\ssy E}{\sf E}^1)_{\ssy 0,D}\\
\leq&\,\dtau^{-1}\,\|\sigma_1\|_{\ssy 0,D}^2
+2\,\dtau\,|\mu|\,\|w_0\|_{\ssy 0,D}\,\|T_{\ssy E}{\sf E}^1\|_{\ssy 0,D}\\
\leq&\,\dtau^{-1}\,\|\sigma_1\|_{\ssy 0,D}^2
+\tfrac{1}{2}\,\|T_{\ssy E}{\sf E}^1\|_{\ssy 0,D}^2
+2\,\dtau^2\,\mu^2\,\|w_0\|^2_{\ssy 0,D}\\
\end{split}
\end{equation*} 
which, finally, yields
\begin{equation}\label{Nov2015_10}
\|T_{\ssy E}{\sf E}^1\|_{\ssy 0,D}^2
+\dtau\,\|{\sf E}^1\|_{\ssy 0,D}^2\leq\,C\,\left(\,\dtau^2\,\|w_0\|^2_{\ssy 0,D}
+\dtau^{-1}\,\|\sigma_1\|_{\ssy 0,D}^2\,\right).
\end{equation}
Next, we use the Cauchy-Schwarz inequality and
\eqref{ElReg1} to get
\begin{equation}\label{Nov2015_11}
\begin{split}
\|\sigma_m\|_{\ssy 0,D}^2\leq&\,2\,\dtau^3\,
\int_{\ssy\Delta_m}\|\partial_{\tau}w(s,\cdot)\|_{\ssy 0,D}^2\;ds
+2\,\mu^2\,\dtau^3\,
\int_{\ssy\Delta_m}\|T_{\ssy E}(\partial_{\tau}w(s,\cdot))\|_{\ssy 0,D}^2\;ds\\
\leq&\,C\,(\Delta\tau)^{3}\,\int_{\ssy\Delta_m}
\|\partial_{\tau}w(s,\cdot)\|_{\ssy 0,D}^2\,ds,
\quad m=1,\dots,M.\\
\end{split}
\end{equation}
Finally, we use \eqref{Nov2015_9}, \eqref{Nov2015_10}, \eqref{Nov2015_11}
and \eqref{Reggo3} (with $\beta=0$, $\ell=1$, $p=0$) to obtain
\begin{equation*}\label{Nov2015_12}
\begin{split}
\dtau\,\sum_{m=1}^{\ssy M}\|{\sf E}^m\|_{\ssy 0,D}^2
\leq&\,C\,\left(\,
\dtau^2\,\|w_0\|_{\ssy 0,D}^2
+\dtau^{-1}\,\sum_{m=1}^{\ssy M}\|\sigma_m\|_{\ssy 0,D}^2\,\right)\\
\leq&\,C\,\left(\,
\dtau^2\,\|w_0\|_{\ssy 0,D}^2
+\dtau^2\,\int_0^{\ssy T}\|\partial_{\tau}w(s,\cdot)\|_{\ssy 0,D}^2\,ds\,\right)\\
&\leq\,C\,\dtau^2\,\|w_0\|^2_{\ssy{\bfdot H}^2},\\
\end{split}
\end{equation*}
which establishes \eqref{Nov2015_Propo1} for $\theta=1$.
\par
From \eqref{BEDet2a}, \eqref{BEDet2} and \eqref{tiger007} follows that
\begin{equation*}
T_{\ssy B}(W^1-W^0)+\dtau\,W^1=0,
\end{equation*}
\begin{equation*}
T_{\ssy B}(W^m-W^{m-1})+\dtau\,\left(\,W^m+\mu\,T_{\ssy E}W^{m-1}\,\right)=0,
\quad m=2,\dots,M.
\end{equation*}
Taking the $L^2(D)-$inner product of both sides of the first 
equation above with $W^1$ and of the second one with
$W^m$, and then applying \eqref{TB-prop1}, \eqref{innerproduct}
and the arithmetic mean inequality, we obtain
\begin{equation}\label{Nov2015_20}
\|T_{\ssy E}W^1\|_{\ssy 0,D}^2 -\|T_{\ssy E}W^0\|_{\ssy
0,D}^2+2\,\dtau\,\|W^1\|_{\ssy 0,D}^2\leq\,0,
\end{equation}
\begin{equation}\label{Nov2015_21}
\|T_{\ssy E}W^m\|_{\ssy 0,D}^2 -\|T_{\ssy E}W^{m-1}\|_{\ssy 0,D}^2
+\dtau\,\|W^m\|_{\ssy 0,D}^2\leq\,\mu^2\,\dtau\,\|T_{\ssy E}W^{m-1}\|_{\ssy
0,D}^2,\quad m=2,\dots,M.
\end{equation}
The inequalities \eqref{Nov2015_20} and \eqref{Nov2015_21}, easily, yield that
\begin{equation*}
\|T_{\ssy E}W^m\|_{\ssy
0,D}^2\leq\,(1+\mu^2\,\Delta\tau)\,\|T_{\ssy E}W^{m-1}\|^2_{\ssy
0,D},\quad m=1,\dots,M,
\end{equation*}
from which, after the use of a standard discrete Gronwall argument, we arrive at
\begin{equation}\label{Nov2015_22}
\max_{0\leq{m}\leq{\ssy M}}\|T_{\ssy E}W^m\|_{\ssy 0,D}^2
\leq\,C\,\|T_{\ssy E}W^0\|_{\ssy 0,D}^2.
\end{equation}
We sum both sides of \eqref{Nov2015_21} with respect to $m$, from $2$ up to $M$,
and then use \eqref{Nov2015_22}, to have
\begin{equation}\label{Nov2015_23}
\begin{split}
\dtau\,\sum_{m=2}^{\ssy M}\|W^m\|_{\ssy
0,D}^2\leq&\,\|T_{\ssy E}W^1\|_{\ssy0,D}^2
+\mu^2\,\dtau\,\sum_{m=1}^{\ssy M-1}\|T_{\ssy E}W^m\|_{\ssy0,D}^2\\
\leq&\,C\,\left(\,\|T_{\ssy E}W^1\|_{\ssy0,D}^2
+\|T_{\ssy E}W^0\|_{\ssy0,D}^2\,\right).\\
\end{split}
\end{equation} 
Thus, using \eqref{Nov2015_23}, \eqref{Nov2015_20},
\eqref{BEDet1}, \eqref{ElReg1}  and
\eqref{minus_equiv} we obtain
\begin{equation}\label{Nov2015_24}
\begin{split}
\dtau\,\sum_{m=1}^{\ssy M}\|W^m\|_{\ssy 0,D}^2
\leq&\,C\,\left(\,\|T_{\ssy E}W^1\|_{\ssy 0,D}^2+
\dtau\,\|W^1\|_{\ssy 0,D}^2+\|T_{\ssy E}w_0\|_{\ssy 0,D}^2\,\right)\\
\leq&\,C\,\|T_{\ssy E}w_0\|_{\ssy 0, D}^2\\
\leq&\,C\,\|w_0\|_{\ssy -2, D}^2\\
\leq&\,C\,\|w_0\|_{\ssy{\bfdot H}^{-2}}^2.\\
\end{split}
\end{equation}
In addition we have
\begin{equation*}
\begin{split}
\dtau\,\sum_{m=1}^{\ssy M}\|w^m\|_{\ssy 0,D}^2
=&\,\sum_{m=1}^{\ssy M}\,\int_{\ssy D}\,
\left(\,\int_{\ssy\Delta_m}\partial_{\tau}\left[\,(\tau-\tau_{m-1})
\,w^2(\tau,x)\,\right]\,d\tau\,\right)\,dx\\
=&\,\sum_{m=1}^{\ssy M}\,\int_{\ssy D}\,
\left(\int_{\ssy\Delta_m}\left[\,w^2(\tau,x)
+2\,(\tau-\tau_{m-1})\,w_{\tau}(\tau,x)\,w(\tau,x)\,\right]
\,d\tau\right)\,dx\\
\leq&\,\sum_{m=1}^{\ssy M}\int_{\ssy\Delta_m}\,
\left(\,2\,\|w(\tau,\cdot)\|_{\ssy 0,D}^2 +(\tau-\tau_{m-1})^2\,
\|w_{\tau}(\tau,\cdot)\|_{\ssy 0,D}^2\,\right)\;d\tau\\
\leq&\,2\,\int_0^{\ssy T}\|w(\tau,\cdot)\|_{\ssy 0,D}^2\;d\tau
+\int_0^{\ssy T}\tau^2\,\|w_{\tau}(\tau,\cdot)\|_{\ssy 0,D}^2\,d\tau,
\end{split}
\end{equation*}
which, along with \eqref{Reggo3} (with $(\beta,\ell,p)=(0,0,0)$
and $(\beta,\ell,p)=(2,1,0)$), yields
\begin{equation}\label{Nov2015_25}
\dtau\,\sum_{m=1}^{\ssy M}\|w^m\|_{\ssy 0,D}^2
\leq\,C\,\|w_0\|^2_{\ssy {\bfdot H}^{-2}}.
\end{equation}
Thus, \eqref{Nov2015_24} and \eqref{Nov2015_25}
establish \eqref{Nov2015_Propo1} for $\theta=0$.
\par
Finally, the estimate \eqref{Nov2015_Propo1} follows by interpolation,
since it is valid for $\theta=1$ and $\theta=0$.
\end{proof}
\par
We close this section by deriving, for later use, the following a priori bound.
%
%
%
%
\begin{lemma}
Let $(W^m)_{m=0}^{\ssy M}$ be the time-discrete approximations defined
by \eqref{BEDet1}--\eqref{BEDet2}. Then,
there exist a constant $C>0$, independent of $\dtau$, such that
\begin{equation}\label{October2010_A}
\left(\,\dtau\,\sum_{m=1}^{\ssy M}
\|\partial_x^3W^m\|_{\ssy 0,D}^2\,\right)^{\frac{1}{2}}
\leq\,C\,\|w_0\|_{\ssy{\bfdot H}^{1}}
\quad\forall\,w_0\in{\bfdot H}^1(D).
\end{equation}
\end{lemma}
%
%
%
%
%
%
%
\begin{proof}
In the sequel,  we will use the symbol $C$ to denote a generic constant that is
independent of $\dtau$ and may changes value from one line to the other.
\par
Taking the $(\cdot,\cdot)_{\ssy 0,D}-$inner product of
\eqref{BEDet2} with $\partial_x^2W^m$
and of \eqref{BEDet2a} with $\partial_x^2W^1$,
and then integrating by parts, we obtain
\begin{equation}\label{Dec2015_11th_1}
\big(\partial_xW^1-\partial_xW^0, \partial_xW^1\big)_{\ssy 0,D}
+\dtau\,\|\partial_x^3W^1\|_{\ssy 0,D}^2=0,
\end{equation}
\begin{equation}\label{Dec2015_11th_2}
\big(\partial_xW^m-\partial_xW^{m-1}, \partial_xW^m\big)_{\ssy
0,D} +\Delta\tau\,\left[\,\|\partial_x^3W^m\|_{\ssy 0,D}^2
+\mu\,(\partial_x^3W^{m},\partial_xW^{m-1})_{\ssy 0,D}\,\right]=0
\end{equation}
for $m=2,\dots,M$. Using \eqref{innerproduct} and
the arithmetic mean inequality, from
\eqref{Dec2015_11th_1} and \eqref{Dec2015_11th_2}
follows that
\begin{equation}\label{Dec2015_11th_3}
\|\partial_xW^1\|^2_{\ssy 0,D}-\|\partial_xW^0\|_{\ssy
0,D}^2+2\,\dtau\,\|\partial_x^3W^1\|_{\ssy 0,D}^2\leq0,
\end{equation}
\begin{equation}\label{Dec2015_11th_4}
\|\partial_xW^m\|^2_{\ssy 0,D}-\|\partial_xW^{m-1}\|_{\ssy
0,D}^2+\dtau\,\|\partial_x^3W^m\|_{\ssy 0,D}^2\leq
\,\dtau\,\mu^2\,\|\partial_xW^{m-1}\|_{\ssy 0,D}^2,\quad
m=2,\dots,M.
\end{equation}
Now, \eqref{Dec2015_11th_4} and \eqref{Dec2015_11th_3}, easily, yield that
\begin{equation*}
\|\partial_xW^m\|_{\ssy 0,D}^2\leq\,(1+\mu^2\,\dtau)
\,\|\partial_xW^{m-1}\|_{\ssy 0,D}^2,\quad m=2,\dots,M,
\end{equation*}
which, after a standard induction argument, leads to
\begin{equation}\label{Dec2015_11th_5}
\max_{1\leq{m}\leq{\ssy M}}\|\partial_xW^m\|_{\ssy
0,D}^2\leq\,C\,\|\partial_xW^1\|_{\ssy 1,D}^2.
\end{equation}
After summing both sides of \eqref{Dec2015_11th_4}
with respect to $m$, from $2$ up to $M$, we obtain
\begin{equation*}
\dtau\,\sum_{m=2}^{\ssy M}\|\partial_x^3W^m\|_{\ssy
0,D}^2\leq\|\partial_xW^1\|_{\ssy
0,D}^2+\mu^2\,\dtau\,\sum_{m=1}^{\ssy M-1}
\|\partial_xW^{m}\|_{\ssy 0,D}^2
\end{equation*}
which, after using \eqref{Dec2015_11th_5}, yields
\begin{equation}\label{Dec2015_11th_6}
\dtau\,\sum_{m=1}^{\ssy M}\|\partial_x^3W^m\|_{\ssy
0,D}^2\leq\,C\,\left(\,
\|\partial_xW^1\|_{\ssy 0,D}^2
+\dtau\,\|\partial_x^3W^1\|_{\ssy 0,D}^2\,\right).
\end{equation}
Finally, we combine \eqref{Dec2015_11th_6}, \eqref{Dec2015_11th_3} and
\eqref{H_equiv} to get
\begin{equation*}
\begin{split}
\dtau\,\sum_{m=1}^{\ssy M}
\|\partial_x^3W^m\|_{\ssy 0,D}^2\leq&\,C\,\|\partial_xW^0\|_{\ssy 0,D}^2\\
\leq&\,C\,\|w_0\|_{\ssy 1,D}^2\\
\leq&\,C\,\|w_0\|_{\ssy{\bfdot H}^1}^2,\\
\end{split}
\end{equation*}
which, easily, yields \eqref{October2010_A}.
\end{proof}
%
%
%
%

%
\subsection{Fully-Discrete Approximations}
The modified-IMEX time-stepping method along with a finite element space
discretization yields a fully discrete method for the approximation of the
solution to the deterministic problem \eqref{Det_Parab}. The method begins by setting
\begin{equation}\label{IMEXFD_1}
W_h^0:=P_hw_0
\end{equation}
and specifing $W_h^1\in{\sf M}_h^r$ such that
\begin{equation}\label{IMEXFD_2}
W_h^1-W_h^0+\dtau\,B_hW_h^1=0.
\end{equation}
Then, for $m=2,\dots,M$, it finds $W_h^m\in{\sf M}_h^r$ such that
\begin{equation}\label{IMEXFD_3}
W_h^m-W_h^{m-1}+\dtau\,\left[\,B_hW_h^m+\mu\,P_h\left(\,
\partial_x^2W_h^{m-1}\,\right)\,\right]=0.
\end{equation}
\par
Adopting the viewpoint that the fully-discrete approximations defined above are approximations of the
time-discrete ones defined in the previous section, we estimate below the corresponding approximation
error in a discrete in time $L^2_t(L^2_x)-$norm. 
%
%
%
\begin{proposition}\label{DetPropo2}
Let $r=2$ or $3$, $(W^m)_{m=0}^{\ssy M}$ be the time discrete approxi\-mations
defined by \eqref{BEDet1}--\eqref{BEDet2}, and
$(W_h^m)_{m=0}^{\ssy M}\subset{\sf M}_h^r$ be the fully discrete
approximations specified in \eqref{IMEXFD_1}--\eqref{IMEXFD_3}.
Then, there exist a constant $C>0$, independent of $\dtau$ and $h$, such that
\begin{equation}\label{Nov2015_Propo2}
\left(\,\dtau\,\sum_{m=1}^{\ssy M}\|W^m-W_h^m\|^2_{\ssy 0,D}\,\right)^{\frac{1}{2}}
\leq\,C\,h^{r \theta}\,\|w_0\|_{\ssy {\bfdot H}^{3\theta-2}}
\quad\forall\,w_0\in {\bfdot H}^1(D),\quad\forall\,\theta\in[0,1].
\end{equation}
%
%
%
%
%
\end{proposition}
%
%
%
%
%
%
\begin{proof}
In the sequel,  we will use the symbol $C$ to denote a generic constant which is
independent of $\dtau$ and $h$, and may changes value from one line to the other.
\par
Let ${\sf Z}^m:=W^m-W_h^m$ for $m=0,\dots,M$. Then, from \eqref{BEDet2a}, \eqref{BEDet2},
\eqref{IMEXFD_2} and \eqref{IMEXFD_3}, we obtain the following error equations:
\begin{equation}\label{IMA_Lost_10}
T_{\ssy B,h}({\sf Z}^1-{\sf Z}^0)+\dtau\,{\sf Z}^1=\dtau\,\xi^1,
\end{equation}
\begin{equation}\label{IMA_Lost_11}
T_{\ssy B,h}({\sf Z}^m-{\sf Z}^{m-1})+\dtau\,\left[\,{\sf Z}^m
+\mu\,T_{\ssy B,h}(\partial_x^2{\sf Z}^{m-1})\,\right]=\dtau\,\xi^m, \quad m=2,\dots,M,
\end{equation}
where
\begin{equation}\label{IMA_Lost_12}
\xi^m:=(T_{\ssy B}-T_{{\ssy B},h})\partial_x^4W^m,\quad m=1,\dots,M.
\end{equation}
\par
Taking the $L^2(D)-$inner product of both sides of
\eqref{IMA_Lost_11} with ${\sf Z}^m$, we obtain
\begin{equation*}
\begin{split}
(T_{\ssy B,h}({\sf Z}^m-{\sf Z}^{m-1}),{\sf Z}^m)_{\ssy 0,D}
+\dtau\,\|{\sf Z}^m\|^2_{\ssy 0,D}=&\,-\mu\,\dtau\,\left(T_{\ssy
B,h}(\partial_x^2{\sf Z}^{m-1}),{\sf Z}^m\right)_{\ssy 0,D}\\
&+\Delta\tau\,(\xi^m,{\sf Z}^m)_{\ssy 0,D},
\quad m=2,\dots,M,\\
\end{split}
\end{equation*}
which, along with \eqref{adjo2} and \eqref{innerproduct}, yields
\begin{equation}\label{Kavalos15th_Veniz2}
\begin{split}
\|\partial_x^2(T_{\ssy B,h}{\sf Z}^m)\|_{\ssy 0,D}^2
-\|\partial_x^2(T_{\ssy B,h}{\sf Z}^{m-1})\|_{\ssy
0,D}^2&\,+\|\partial_x^2\left(T_{\ssy B,h}\left({\sf Z}^m-{\sf Z}^{m-1}\right)\right)\|_{\ssy 0,D}^2\\
&\,\hskip1.5truecm+2\,\dtau\,\|{\sf Z}^m\|^2_{\ssy 0,D}
={\mathcal A}_1^m+{\mathcal A}_2^m\\
\end{split}
\end{equation}
for $m=2,\dots,M$, where
\begin{equation*}
\begin{split}
{\mathcal A}^m_1:=&\,2\,\dtau\,(\xi^m,{\sf Z}^m)_{\ssy 0,D},\\
{\mathcal A}_2^m:=&\,-2\,\mu\,\dtau\,\left(T_{\ssy
B,h}\left(\partial_x^2{\sf Z}^{m-1}\right),{\sf Z}^m\right)_{\ssy 0,D}.\\
\end{split}
\end{equation*}
Using \eqref{adjo2},  integration by parts, 
the Cauchy-Schwarz inequality, the arithmetic mean inequality, we have
\begin{equation}\label{Dec2015_2}
{\mathcal A}_1^m\leq\dtau\,\left(\,\|{\sf Z}^m\|_{\ssy 0,D}^2+\|\xi^m\|^2_{\ssy 0,D}\,\right)
\end{equation}
and
\begin{equation}\label{Dec2015_1}
\begin{split}
{\mathcal A}_2^m=&\,-2\,\mu\,\dtau\,(\partial_x^2{\sf Z}^{m-1},T_{\ssy B,h}{\sf Z}^m)_{\ssy 0,D}\\
=&\,-2\,\mu\,\dtau\,({\sf Z}^{m-1},\partial_x^2(T_{\ssy B,h}{\sf Z}^m))_{\ssy 0,D}\\
=&\,-2\,\mu\,\dtau\,\left({\sf Z}^{m-1},\partial_x^2
\left(T_{\ssy B,h}\left({\sf Z}^m-{\sf Z}^{m-1}\right)\right)\right)_{\ssy 0,D}\\
&\quad\quad-2\,\mu\,\dtau\,({\sf Z}^{m-1},
\partial_x^2(T_{\ssy B,h}{\sf Z}^{m-1}))_{\ssy 0,D}\\
\leq&\,2\,|\mu|\,\dtau\,\|{\sf Z}^{m-1}\|_{\ssy 0,D}\,
\left\|\partial_x^2\left(T_{\ssy B,h}\left({\sf Z}^m-{\sf Z}^{m-1}\right)\right)\right\|_{\ssy 0,D}\\
&\quad\quad+2\,|\mu|\,\dtau\,\|{\sf Z}^{m-1}\|_{\ssy 0,D}\,
\left\|\partial_x^2\left(T_{\ssy B,h}{\sf Z}^{m-1}\right)\right\|_{\ssy 0,D}\\
\leq&\,\dtau^2\,\mu^2\,\|{\sf Z}^{m-1}\|^2_{\ssy 0,D}
+\|\partial_x^2\left(T_{\ssy B,h}\left({\sf Z}^m-{\sf Z}^{m-1}\right)\right)\|_{\ssy 0,D}^2\\
&\quad\quad+\tfrac{\dtau}{2}\,\|{\sf Z}^{m-1}\|^2_{\ssy 0,D}
+2\,\dtau\,\mu^2\,\|\partial_x^2(T_{\ssy B,h}{\sf Z}^{m-1})\|_{\ssy 0,D}^2,
\quad m=2,\dots,M.\\
\end{split}
\end{equation}
Now, we combine \eqref{Kavalos15th_Veniz2},
\eqref{Dec2015_2} and \eqref{Dec2015_1} to get
\begin{equation}\label{Dec2015_3}
\begin{split}
\|\partial_x^2(T_{\ssy B,h}{\sf Z}^m)\|_{\ssy 0,D}^2
+\dtau\,\|{\sf Z}^m\|^2_{\ssy 0,D}\leq&\,
\|\partial_x^2(T_{\ssy B,h}{\sf Z}^{m-1})\|_{\ssy 0,D}^2
+\tfrac{\dtau}{2}\,\|{\sf Z}^{m-1}\|_{\ssy 0,D}^2+\dtau\,\|\xi^m\|_{\ssy 0,D}^2\\
&+2\,\dtau\,\mu^2\,\left(\,\|\partial_x^2(T_{\ssy B,h}{\sf Z}^{m-1})\|_{\ssy 0,D}^2
+\dtau\,\|{\sf Z}^{m-1}\|_{\ssy 0,D}^2\,\right)\\
\end{split}
\end{equation}
for $m=2,\dots,M$.
Let
$\Upsilon^{\ell}:=\|\partial_x^2(T_{\ssy B,h}{\sf Z}^{\ell})\|_{\ssy
0,D}^2+\dtau\,\|{\sf Z}^{\ell}\|_{\ssy 0,D}^2$ for $\ell=1,\dots,M$.
Then \eqref{Dec2015_3} yields
\begin{equation*}
\Upsilon^m\leq(1+2\,\mu^2\,\dtau)\,\Upsilon^{m-1}
+\dtau\,\|\xi^m\|_{\ssy 0,D}^2,\quad m=2,\dots,M,
\end{equation*}
from which, after applying a standard discrete Gronwall argument, we conclude that
\begin{equation}\label{Dec2015_4}
\max_{1\leq{m}\leq{\ssy M}}
\Upsilon^m\leq\,C\left(\,\Upsilon^1
+\dtau\,\sum_{m=2}^{\ssy M}\|\xi^m\|_{\ssy 0,D}^2\,\right).
\end{equation}
Since $T_{\ssy B,h}{\sf Z}^0=0$, after taking the $L^2(D)-$inner product
of both sides of \eqref{IMA_Lost_10} with ${\sf Z}^1$, and then
using \eqref{adjo2} and the arithmetic mean inequality, we obtain
\begin{equation}\label{Dec2015_5}
\|\partial_x^2(T_{\ssy B,h}{\sf Z}^1)\|_{\ssy
0,D}^2+\tfrac{\dtau}{2}\,\|{\sf Z}^1\|_{\ssy
0,D}^2\leq\,\tfrac{\dtau}{2}\,\|\xi^1\|^2_{\ssy 0,D},
\end{equation}
which, along with \eqref{Dec2015_4}, yields
\begin{equation}\label{Dec2015_6}
\max_{1\leq{m}\leq{\ssy M}}\Upsilon^m\leq\,C\,\dtau\,\sum_{m=1}^{\ssy M}
\|\xi^m\|_{\ssy 0,D}^2.
\end{equation}
Now, summing both sides of \eqref{Dec2015_3} with respect to $m$,
from $2$ up to $M$, we obtain
\begin{equation*}
\begin{split}
\dtau\,\sum_{m=2}^{\ssy M}\|{\sf Z}^m\|_{\ssy 0,D}^2
\leq&\,\|\partial_x^2(T_{\ssy B,h}{\sf Z}^1)\|_{\ssy 0,D}^2
+\tfrac{\dtau}{2}\,\sum_{m=1}^{\ssy M-1}\|{\sf Z}^m\|_{\ssy 0,D}^2\\
&\quad+\dtau\,\sum_{m=2}^{\ssy M}\|\xi^m\|_{\ssy 0,D}^2
+2\,\mu^2\,\dtau\,\sum_{m=1}^{\ssy M-1}\Upsilon^{m},\\
\end{split}
\end{equation*}
which, along with \eqref{Dec2015_6}, yields
\begin{equation}\label{Dec2015_7}
\begin{split}
\tfrac{\dtau}{2}\,\sum_{m=1}^{\ssy M}\|{\sf Z}^m\|_{\ssy 0,D}^2
\leq&\|\partial_x^2(T_{\ssy B,h}{\sf Z}^1)\|_{\ssy 0,D}^2
+\dtau\,\|{\sf Z}^1\|_{\ssy 0,D}^2\\
&\quad+\dtau\,\sum_{m=2}^{\ssy M}\|\xi^m\|_{\ssy 0,D}^2
+2\,\mu^2\,\dtau\,\sum_{m=1}^{\ssy M-1}\Upsilon^m\\
\leq&\,C\,\left(\,\max_{1\leq{m}\leq{\ssy M}-1}\Upsilon^m
+\dtau\,\sum_{m=2}^{\ssy M}\|\xi^m\|_{\ssy 0,D}^2\,\right)\\
\leq&\,C\,\dtau\,\sum_{m=1}^{\ssy M}\|\xi^m\|_{\ssy 0,D}^2.\\
\end{split}
\end{equation}
Combining \eqref{Dec2015_7}, \eqref{IMA_Lost_12},
\eqref{ARA1} and \eqref{October2010_A}, we obtain
\begin{equation}\label{Dec2015_8}
\begin{split}
\dtau\,\sum_{m=1}^{\ssy M}\|{\sf Z}^m\|_{\ssy 0,D}^2
\leq&\,C\,h^{2 r}\,\dtau\,\sum_{m=1}^{\ssy M}\|\partial_x^3W^m\|_{\ssy 0,D}^2\\
\leq&\,C\,h^{2 r}\,\|w_0\|^2_{\ssy{\bfdot H}^1}.\\
\end{split}
\end{equation}
%
%
%
%
Thus, \eqref{Dec2015_8} yields \eqref{Nov2015_Propo2} for $\theta=1$.
\par
From \eqref{IMEXFD_2} and \eqref{IMEXFD_3} we conclude that
\begin{equation*}
T_{\ssy B,h}(W_h^1-W_h^0)+\dtau\,W_h^1=0,
\end{equation*}
\begin{equation*}
T_{\ssy B,h}(W_h^m-W_h^{m-1})+\dtau\,W_h^m
=-\mu\,\dtau\,T_{\ssy B,h}(\partial_x^2W_h^{m-1}),
\quad m=2,\dots,M.
\end{equation*}
Taking the $L^2(D)-$inner product of both sides of the first 
equation above with $W_h^1$ and of the second one with
$W_h^m$, and then applying \eqref{adjo2} and
\eqref{innerproduct}, we obtain
\begin{equation}\label{Dec2015_20}
\|\partial_x^2(T_{\ssy B,h}W_h^1)\|_{\ssy 0,D}^2
-\|\partial_x^2(T_{\ssy B,h}W_h^0)\|_{\ssy 0,D}^2
+2\,\dtau\,\|W_h^1\|_{\ssy 0,D}^2\leq\,0,
\end{equation}
\begin{equation}\label{Dec2015_21}
\begin{split}
\|\partial_x^2(T_{\ssy B,h}W_h^m)\|_{\ssy 0,D}^2
&\,+\|\partial_x^2(T_{\ssy B,h}(W_h^m-W_h^{m-1}))\|_{\ssy 0,D}^2\\
&\,+2\dtau\,\|W_h^m\|_{\ssy 0,D}^2=\|\partial_x^2(T_{\ssy B,h}W_h^{m-1})\|_{\ssy 0,D}^2
+{\mathcal A}_3^m,\quad m=2,\dots,M,\\
\end{split}
\end{equation}
where
\begin{equation*}
{\mathcal A}_3^m:=-2\,\mu\,\dtau\,\left(T_{\ssy B,h}
\left(\partial_x^2W_h^{m-1}\right),W_h^{m}\right)_{\ssy 0,D}.
\end{equation*}
Using \eqref{adjo2},  integration by parts, the Cauchy-Schwarz inequality,
and the arithmetic mean inequality, we have
\begin{equation}\label{Dec2015_11th_7}
\begin{split}
{\mathcal A}_3^m
=&\,-2\,\mu\,\dtau\,\left(W_h^{m-1},\partial_x^2\left(T_{\ssy B,h}W_h^{m}\right)\right)_{\ssy 0,D}\\
=&\,-2\,\mu\,\dtau\,\left(W_h^{m-1},\partial_x^2
\left(T_{\ssy B,h}\left(W_h^m-W_h^{m-1}\right)\right)\right)_{\ssy 0,D}\\
&\quad-2\,\mu\,\dtau\,(W_h^{m-1},
\partial_x^2(T_{\ssy B,h}W_h^{m-1}))_{\ssy 0,D}\\
\leq&\,\dtau^2\,\mu^2\,\|W_h^{m-1}\|^2_{\ssy 0,D}
+\|\partial_x^2\left(T_{\ssy B,h}\left(W_h^m-W_h^{m-1}\right)\right)\|_{\ssy 0,D}^2\\
&\quad+\tfrac{\dtau}{2}\,\|W_h^{m-1}\|^2_{\ssy 0,D}
+2\,\dtau\,\mu^2\,\|\partial_x^2(T_{\ssy B,h}W_h^{m-1})\|_{\ssy 0,D}^2,
\quad m=2,\dots,M.\\
\end{split}
\end{equation}
Combining \eqref{Dec2015_21} and \eqref{Dec2015_11th_7}, we arrive at
\begin{equation}\label{Dec2015_11th_9}
\begin{split}
\|\partial_x^2(T_{\ssy B,h}W_h^m)\|_{\ssy 0,D}^2&\,+2\,\dtau\,\|W^m_h\|_{\ssy 0,D}^2
\leq\|\partial_x^2(T_{\ssy B,h}W_h^{m-1})\|_{\ssy 0,D}^2
+\tfrac{\dtau}{2}\,\|W^{m-1}_h\|_{\ssy 0,D}^2\\
&+2\,\dtau\,\mu^2\,\left(\,\|\partial_x^2(T_{\ssy B,h}W_h^{m-1})\|_{\ssy 0,D}^2
+\dtau\,\|W_h^{m-1}\|_{\ssy 0,D}^2\,\right), 
\quad m=2,\dots,M.\\
\end{split}
\end{equation}
Let$\Upsilon_h^{\ell}:=\|\partial_x^2(T_{\ssy B,h}W_h^{\ell})\|_{\ssy 0,D}^2
+\dtau\,\|W_h^{\ell}\|_{\ssy 0,D}^2$ for $\ell=1,\dots,M$. Then,
we use \eqref{Dec2015_20}, \eqref{IMEXFD_1}, \eqref{TB_bound},
\eqref{minus_equiv} and \eqref{Dec2015_11th_9} to obtain
\begin{equation}\label{Dec2015_17th_a}
\begin{split}
\Upsilon_h^1\leq&\,\|\partial_x^2(T_{\ssy B,h}P_hw_0)\|_{\ssy 0,D}^2\\
\leq&\,\|\partial_x^2(T_{\ssy B,h}w_0)\|_{\ssy 0,D}^2\\
\leq&\,\|w_0\|_{\ssy -2,D}^2\\
\leq&\,\|w_0\|_{\ssy{\bfdot H}^{-2}}^2
\end{split}
\end{equation}
and
\begin{equation}\label{Dec2015_17th_b}
\Upsilon_h^{m}\leq(1+2\,\mu^2\dtau)\,\Upsilon_h^{m-1},
\quad m=2,\dots,M.
\end{equation}
From \eqref{Dec2015_17th_b}, after the application of
a standard discrete Gronwall argument and the use of \eqref{Dec2015_17th_a},
we conclude that
\begin{equation}\label{Dec2015_100}
\begin{split}
\max_{1\leq{m}\leq{\ssy M}}
\Upsilon_h^m\leq&\,C\,\Upsilon_h^1\\
\leq&\,C\,\|w_0\|_{\ssy{\bfdot H}^{-2}}^2.
\end{split}
\end{equation}
Summing both sides of \eqref{Dec2015_11th_9} with respect to $m$,
from $2$ up to $M$, we have
\begin{equation*}
\begin{split}
\dtau\,\sum_{m=2}^{\ssy M}\|W_h^m\|_{\ssy 0,D}^2
\leq&\,\|\partial_x^2(T_{\ssy B,h}W_h^1)\|_{\ssy 0,D}^2
+\tfrac{\dtau}{2}\,\sum_{m=1}^{\ssy M-1}\|W_h^m\|_{\ssy 0,D}^2
+2\,\mu^2\,\dtau\,\sum_{m=1}^{\ssy M-1}\Upsilon_h^m,
\end{split}
\end{equation*}
which, along with \eqref{Dec2015_100}, yields
\begin{equation}\label{Dec2015_101}
\begin{split}
\tfrac{\dtau}{2}\,\sum_{m=1}^{\ssy M}\|W_h^m\|_{\ssy 0,D}^2
\leq&\Upsilon_h^1
+2\,\mu^2\,\dtau\,\sum_{m=1}^{\ssy M-1}\Upsilon_h^m\\
\leq&\,C\,\|w_0\|_{\ssy{\bfdot H}^{-2}}^2.\\
\end{split}
\end{equation}
Thus, \eqref{Dec2015_101} and \eqref{Nov2015_24}
yield \eqref{Nov2015_Propo2} for $\theta=0$.
\par
Thus, the error estimate \eqref{Nov2015_Propo2} follows by interpolation,
since it holds for $\theta=1$ and $\theta=0$.
\end{proof}
%
%
%
%

%
%
\section{Convergence analysis of the IMEX finite element method}\label{Stoch_Section}
In order to estimate the approximation error of the IMEX finite element method
given in Section~\ref{The_Numerical_Method},  we use, as a tool, the
corresponding IMEX time-discrete approximations of $\uu$, which
are defined first by setting
\begin{equation}\label{BackE1}
\UU^0:=0
\end{equation}
and then, for $m=1,\dots,M$, by seeking $\UU^m\in {\bfdot H}^4(D)$ such that
\begin{equation}\label{BackE2}
\UU^m-\UU^{m-1}+\dtau\,\left(\,\partial_x^4\UU^m
+\mu\,\partial_x^2\UU^{m-1}\right)
=\int_{\ssy\Delta_m}\partial_x{\mathcal W}\,d\tau\quad\text{\rm a.s.}.
\end{equation}
Thus, we split the total error of the IMEX finite element method as follows
\begin{equation}\label{split_split}
\max_{0\leq{m}\leq{\ssy M}}\left({\mathbb E}\left[\,
\|\uu^m-\UU_h^m\|_{\ssy 0,D}^2\,\right]\right)^{\frac{1}{2}}
\leq\max_{0\leq{m}\leq{\ssy M}}{\mathcal E}_{\ssy\rm TDR}^m
+\max_{0\leq{m}\leq{\ssy M}}{\mathcal E}^m_{\ssy\rm SDR},
\end{equation}
where $\uu^m:=\uu(\tau_m,\cdot)$,
${\mathcal E}^m_{\ssy\rm TDR}:=
\left({\mathbb E}\left[\|\uu^m-\UU^m\|_{\ssy 0,D}^2 \right]\right)^{\ssy 1/2}$ is the {\sl time discretization error} 
at $\tau_m$, and
${\mathcal E}_{\ssy\rm SDR}^m:=
\left({\mathbb E}\left[\|\UU^m-\UU^m_h\|_{\ssy 0,D}^2 \right]\right)^{\ssy 1/2}$ is the {\sl space discretization error} 
at $\tau_m$.
\subsection{Estimating the time discretization error}
\par
The convergence estimate of Proposition~\ref{DetPropo1} is the main tool in providing
a discrete in time $L^{\infty}_t(L^2_{\ssy P}(L^2_x))$ error estimate
of the time-discretization error (cf. \cite{YubinY05}, \cite{KZ2010}, \cite{KZ2013b}).
%
%
%
%
\begin{proposition}\label{TimeDiscreteErr1}
Let $\uu$ be the solution to \eqref{AC2} and
$(\UU^m)_{m=0}^{\ssy M}$ be the  time-discrete approximations of $\uu$
defined by \eqref{BackE1}--\eqref{BackE2}.
Then, there exists a constant  ${\widehat c}_{\ssy\rm TDR}$,
independent of $\Delta{t}$, ${\sf M}$ and $\dtau$, such that
\begin{equation}\label{ElPasso}
\max_{0\leq m \leq {\ssy M}}{\mathcal E}_{\ssy{\rm TDR}}^{m}
\leq\,{\widehat c}_{\ssy\rm TDR}\,\epsilon^{-\frac{1}{2}}
\,\dtau^{\frac{1}{8}-\epsilon}
\quad\forall\,\epsilon\in\left(0,\tfrac{1}{8}\right].
\end{equation}
%
%
\end{proposition}
%
%
%
%
%
%
%
%
%
%
\begin{proof}
In the sequel, we will use the symbol $C$ to denote a generic constant that is independent
of $\Delta{t}$, ${\sf M}$ and $\dtau$, and may changes value from one line to the other.
\par
First, we introduce some notation by letting ${\sf I}:L^2(D)\to L^2(D)$ be the identity operator,
${\sf Y}:H^2(D)\rightarrow L^2(D)$ be the differential operator
${\sf Y}:={\sf I}-\dtau\,\mu\,\partial_x^2$,
and
${\sf\Lambda}:L^2(D)\to{\bfdot H}^4(D)$ be the inverse
elliptic operator ${\sf\Lambda}:=({\sf I}+\dtau\,\partial_x^4)^{-1}$.
%
%
Then, for $m=1,\dots,M$, we define the operator
${\sf Q}^m:L^2(D)\to{\bfdot H}^4(D)$ by
${\sf Q}^m:=({\sf\Lambda}\circ{\sf Y})^{m-1}\circ{\sf\Lambda}$.
Also, for given $w_0\in{\bfdot H}^2(D)$, let
$({\mathcal S}_{\ssy{\Delta\tau}}^m(w_0))_{m=0}^{\ssy M}$
be time-discrete approximations of the solution to the deterministic
problem \eqref{Det_Parab}, defined by \eqref{BEDet1}--\eqref{BEDet2}. 
Then, using a simple induction argument, we conclude that
\begin{equation}\label{EMP_2010_1}
{\mathcal S}_{\ssy{\Delta\tau}}^m(w_0)={\sf Q}^{m}(w_0),\quad m=1,\dots,M.
\end{equation}
\par
Let $m\in\{1,\dots,M\}$. Applying a simple induction argument on \eqref{BackE2} we
conclude that
\begin{equation*}
\UU^m=\sum_{\ell=1}^{\ssy m} \int_{\ssy\Delta_\ell}
{\sf Q}^{m-\ell+1}\left(\partial_x{\mathcal W}(\tau,\cdot)\right)\,d\tau,
\end{equation*}
which, along with \eqref{SLoad} and \eqref{EMP_2010_1}, yields
\begin{equation}\label{NewLight1}
\begin{split}
%
%
\UU^m=&\,-\tfrac{1}{\Delta{t}}\,\sum_{i=1}^{\ssy{\sf M}}\sum_{n=1}^{\ssy{\sf N}} 
R_i^n\,\lambda_i\,\left(\,\sum_{\ell=1}^{\ssy m}\int_{\ssy\Delta_\ell}
{\mathcal X}_{\ssy T_n}(\tau)
\,{\mathcal S}_{\ssy\Delta\tau}^{m-\ell+1}(\varepsilon_i)\,d\tau\,\right)\\
=&\,-\tfrac{1}{\Delta{t}}\,\sum_{i=1}^{\ssy{\sf M}}\sum_{n=1}^{\ssy{\sf N}} 
R_i^n\,\lambda_i\,\left[\,\int_0^{\ssy T}
{\mathcal X}_{\ssy T_n}(\tau)\,\left(\,\sum_{\ell=1}^{\ssy m}
{\mathcal X}_{\ssy\Delta_{\ell}}(\tau)
\,{\mathcal S}_{\ssy\Delta\tau}^{m-\ell+1}(\varepsilon_i)\,\right)\,d\tau\,\right]\\
=&\,-\tfrac{1}{\Delta{t}}\,\sum_{i=1}^{\ssy{\sf M}}\sum_{n=1}^{\ssy{\sf N}} 
R_i^n\,\lambda_i\,\left[\,\int_{\ssy T_n}
\left(\,\sum_{\ell=1}^{\ssy m}
{\mathcal X}_{\ssy\Delta_{\ell}}(\tau)
\,{\mathcal S}_{\ssy\Delta\tau}^{m-\ell+1}(\varepsilon_i)\,\right)\,d\tau\,\right].\\
\end{split}
\end{equation}
Also, using \eqref{HatUform} and \eqref{SLoad}, and proceeding in similar manner,
we arrive at
\begin{equation}\label{NewLight2}
\begin{split}
\uu^m=&\,\int_0^{\tau_m}{\mathcal S}(\tau_m-\tau)
\,\left(\partial_x{\mathcal W}(\tau,\cdot)\right)\,d\tau\\
=&\,-\tfrac{1}{\Delta{t}}\,\sum_{i=1}^{\ssy{\sf M}}\,\sum_{n=1}^{\ssy{\sf N}}
R_i^n\,\lambda_i
\,\left[\,\int_{\ssy T_n}\left(\,\sum_{\ell=1}^m
{\mathcal X}_{\ssy{\Delta_{\ell}}}(\tau)
\,{\mathcal S}(\tau_m-\tau)
\left(\varepsilon_i\right)\,\right)\,d\tau\,\right].\\
\end{split}
\end{equation}
Thus, using \eqref{NewLight1} and \eqref{NewLight2}
along with Remark~\ref{WRemark}, we obtain
\begin{equation*}
\begin{split}
\left(\,{\mathcal E}^m_{\ssy\rm TDR}\,\right)^2=&\,\tfrac{1}{\Delta{t}}
\sum_{i=1}^{\ssy{\sf M}}\,\sum_{n=1}^{\ssy{\sf N}}
\lambda_i^2\,\int_{\ssy D}\left(\,\int_{\ssy T_n}
\left(\,\sum_{\ell=1}^m
{\mathcal X}_{\ssy{\Delta_{\ell}}}(\tau)
\,\,\left[{\mathcal S}_{\ssy\Delta\tau}^{m-\ell+1}(\varepsilon_i)
-{\mathcal S}(\tau_m-\tau)(\varepsilon_i)\right]\,\right)\,d\tau\,\right)^2\,dx\\
\leq&\,\sum_{i=1}^{\ssy{\sf M}}\,\sum_{n=1}^{\ssy{\sf N}}
\lambda_i^2\,\int_{\ssy D}\int_{\ssy T_n}
\left(\,\sum_{\ell=1}^m{\mathcal X}_{\ssy{\Delta_{\ell}}}(\tau)
\,\left[{\mathcal S}_{\ssy\Delta\tau}^{m-\ell+1}(\varepsilon_i)
-{\mathcal S}(\tau_m-\tau)(\varepsilon_i)\right]
\,\right)^2\,d\tau\,dx\\
\leq&\,\sum_{i=1}^{\ssy{\sf M}}\,
\lambda_i^2\,\int_0^{\ssy T}\int_{\ssy D}
\left(\,\sum_{\ell=1}^m{\mathcal X}_{\ssy{\Delta_{\ell}}}(\tau)
\,\left[{\mathcal S}_{\ssy\Delta\tau}^{m-\ell+1}(\varepsilon_i)
-{\mathcal S}(\tau_m-\tau)(\varepsilon_i)\right]
\,\right)^2\,dx\,d\tau\\
\leq&\,\sum_{i=1}^{\ssy{\sf M}}\,
\lambda_i^2\,\left(\,\sum_{\ell=1}^m\int_{\ssy\Delta_{\ell}}
\|{\mathcal S}_{\ssy\Delta\tau}^{m-\ell+1}(\varepsilon_i)
-{\mathcal S}(\tau_m-\tau)(\varepsilon_i)\|_{\ssy 0,D}^2\,d\tau\,\right),\\
\end{split}
\end{equation*}
which, easily, yields 
\begin{equation}\label{mainerrorF}
{\mathcal E}^m_{\ssy\rm TDR}\leq\,\sqrt{{\mathcal B}_1^m}+\sqrt{{\mathcal
B}_2^m},
\end{equation}
with
\begin{equation*}
\begin{split}
{\mathcal B}_1^m:=&\,\sum_{i=1}^{\ssy{\sf M}}\,
\lambda_i^2\,\left(\,\sum_{\ell=1}^m\dtau\,
\left\|{\mathcal S}_{\ssy\Delta\tau}^{m-\ell+1}(\varepsilon_i)
-{\mathcal S}(\tau_{m-\ell+1})(\varepsilon_i)\right\|_{\ssy 0,D}^2\,\right),\\
{\mathcal B}_2^m:=&\,\sum_{i=1}^{\ssy{\sf M}}\,
\lambda_i^2\,\left(\,\sum_{\ell=1}^m\int_{\ssy\Delta_{\ell}}
\left\|{\mathcal S}(\tau_{m-\ell+1})(\varepsilon_i)
-{\mathcal S}(\tau_m-\tau)(\varepsilon_i)\right\|_{\ssy 0,D}^2\,d\tau\,\right).\\
\end{split}
\end{equation*}
Proceeding as in the proof of Theorem~4.1 in \cite{KZ2013b} we get 
\begin{equation}\label{Ydaspis952}
\sqrt{{\mathcal B}_2^m}\leq\,C\,\dtau^{\frac{1}{8}}.
\end{equation}
Also, using the error estimate \eqref{Nov2015_Propo1} it follows that
\begin{equation*}\label{refo_2}
\begin{split}
\sqrt{{\mathcal B}_1^m}\leq&\,C\,\dtau^{\theta}
\,\left(\,\sum_{i=1}^{\ssy{\sf M}}\lambda_i^2\,\|\varepsilon_i\|^2_{\ssy
{\bfdot H}^{4\theta-2}}\,\right)^{\frac{1}{2}}\\
\leq&\,C\,\dtau^{\theta}\,\left(\,\sum_{i=1}^{\ssy{\sf M}}
\tfrac{1}{\lambda_i^{2-8\theta}}\,\right)^{\frac{1}{2}}\quad\forall\,\theta\in[0,1].\\
\end{split}
\end{equation*}
Setting $\theta=\tfrac{1}{8}-\epsilon$ with $\epsilon\in\left(0,\tfrac{1}{8}\right]$, we have
\begin{equation}\label{trabajito}
\begin{split}
\sqrt{{\mathcal B}_1^m}\leq&\,C\,\dtau^{\frac{1}{8}-\epsilon}
\,\left(\,\sum_{i=1}^{\ssy{\sf M}} \tfrac{1}{i^{1+8\epsilon}}\, \right)^{\half}\\
\leq&\,C\,\dtau^{\frac{1}{8}-\epsilon}\,\left(\,1+\int_1^{\ssy{\sf M}}x^{-1-8\epsilon}\,dx\,\right)^{\half}\\
\leq&\,C\,\dtau^{\frac{1}{8}-\epsilon}\,\epsilon^{-\half}
\,\left(1-\tfrac{1}{{\sf M}^{8\epsilon}}\right)^{\half}.\\
\end{split}
\end{equation}
\par
Thus, the estimate \eqref{ElPasso} follows, easily, as a simple consequence of \eqref{mainerrorF}, \eqref{Ydaspis952}
and \eqref{trabajito}.
\end{proof}
%
%
%
%
\subsection{Estimating the space discretization error}
\par
The outcome of Proposition~\ref{DetPropo2} will be used
below in the derivation of a discrete in time $L^{\infty}_t(L^2_{\ssy P}(L^2_x))$
error estimate of the space discretization error (cf. \cite{YubinY05},
\cite{KZ2010}, \cite{KZ2013b}).
%
%
%
\begin{proposition}\label{Tigrakis}
Let $r=2$ or $3$, $(\UU_h^m)_{m=0}^{\ssy M}$ be the fully dicsrete approximations
defined by  \eqref{FullDE1}--\eqref{FullDE2} and $(\UU^m)_{m=0}^{\ssy M}$
be the time discrete approximations defined by \eqref{BackE1}--\eqref{BackE2}.
Then, there exists a constant ${\widehat c}_{\ssy\rm SDR}>0$,
independent of ${\sf M}$, $\Delta{t}$, $\dtau$ and $h$,
such that
\begin{equation}\label{Lasso1}
\max_{0\leq{m}\leq {\ssy M}}{\mathcal E}^m_{\ssy\rm SDR}
\leq\,{\widehat c}_{\ssy\rm SDR}
\,\epsilon^{-\frac{1}{2}} \,\,\,h^{\frac{r}{6}-\epsilon}
\quad\forall\,\epsilon\in\left(0,\tfrac{r}{6}\right].
\end{equation}
%
%
%
\end{proposition}
%
%
%
%
%
%
%
%
%
\begin{proof}
In the sequel, we will use the symbol $C$ to denote a generic constant that is independent
of $\Delta{t}$, ${\sf M}$, $\dtau$ and $h$, and may changes value from one line to the other.
\par
Let us denote by ${\sf I}:L^2(D)\to L^2(D)$ the identity operator,
by ${\sf Y}_h:{\sf M}_h^r\rightarrow{\sf M}_h^r$ the discrete differential
operator ${\sf Y}_h:={\sf I}-\mu\,\dtau\,(P_h\circ\partial_x^2)$,
${\sf\Lambda}_h:L^2(D)\to {\sf M}^r_h$ be the inverse discrete elliptic
operator  ${\sf\Lambda}_h:=(I+\Delta\tau\,B_h)^{-1}\circ P_h$.
Then,  for $m=1,\dots,M$, we define the auxiliary operator ${\sf Q}_h^m:L^2(D)\to{\sf M}_h^r$
by ${\sf Q}^m_h:=({\sf\Lambda}_h\circ{\sf Y}_h)^{m-1}\circ{\sf\Lambda}_h$.
Also, for given $w_0\in{\bfdot H}^2(D)$,
let $({\mathcal S}_{h}^m(w_0))_{m=0}^{\ssy M}$
be fully discrete discrete approximations of the solution to the deterministic
problem \eqref{Det_Parab}, defined by \eqref{IMEXFD_1}--\eqref{IMEXFD_3}. 
Then, using a simple induction argument, we conclude that
\begin{equation}\label{EMP_2010_111}
{\mathcal S}_{h}^m(w_0)={\sf Q}_h^{m}(w_0),\quad m=1,\dots,M.
\end{equation}
\par
Let $m\in\{1,\dots,M\}$. Using a simple induction argument on \eqref{FullDE2},
\eqref{SLoad} and \eqref{EMP_2010_111}, we conclude that
\begin{equation}\label{NewLight111}
\begin{split}
\UU_h^m=&\,\sum_{\ell=1}^{\ssy m} \int_{\ssy\Delta_\ell}
{\sf Q}_h^{m-\ell+1}\left(\partial_x{\mathcal W}(\tau,\cdot)\right)\,d\tau\\
=&\,-\tfrac{1}{\Delta{t}}\,\sum_{i=1}^{\ssy{\sf M}}\sum_{n=1}^{\ssy{\sf N}} 
R_i^n\,\lambda_i\,\left[\,\int_{\ssy T_n}
\left(\,\sum_{\ell=1}^{\ssy m}
{\mathcal X}_{\ssy\Delta_{\ell}}(\tau)
\,{\mathcal S}_{h}^{m-\ell+1}(\varepsilon_i)\,\right)\,d\tau\,\right].\\
\end{split}
\end{equation}
After, using \eqref{NewLight111}, \eqref{NewLight1} and Remark~\ref{WRemark},
and proceeding as in the proof of Proposition~\ref{TimeDiscreteErr1}, we arrive at
\begin{equation*}
{\mathcal E}^m_{\ssy\rm SDR}\leq\left[\,\sum_{i=1}^{\ssy{\sf M}}\,
\lambda_i^2\,\left(\,\sum_{\ell=1}^m\dtau\,
\|{\mathcal S}_{\ssy\Delta\tau}^{m-\ell+1}(\varepsilon_i)
-{\mathcal S}^{m-\ell+1}_{h}(\varepsilon_i)\|_{\ssy 0,D}^2\,d\tau\,\right)\,\right]^{\half},
\end{equation*}
which, along \eqref{Nov2015_Propo2}, yields
\begin{equation}\label{refo_1}
\begin{split}
%
%
{\mathcal E}^m_{\ssy\rm SDR}\leq&\,C\,h^{r\theta}
\,\left(\,\sum_{i=1}^{\ssy{\sf M}}\lambda_i^2\,\|\varepsilon_i\|^2_{\ssy
{\bfdot H}^{3\theta-2}}\,\right)^{\half}\\
%
%
\leq&\,C\,h^{r\theta}\,\left(\,
\,\sum_{i=1}^{\ssy{\sf M}}\tfrac{1}{\lambda_i^{2-6\theta}}\,\right)^{\half}
\quad\forall\,\theta\in[0,1].\\
\end{split}
\end{equation}
Setting $\theta=\tfrac{1}{6}-\delta$ with $\delta\in\left(0,\tfrac{1}{6}\right]$, we have
\begin{equation*}
\begin{split}
{\mathcal E}^m_{\ssy\rm SDR}\leq&\,C\,h^{\frac{r}{6}-r\delta}
\,\left(\,\sum_{i=1}^{\ssy{\sf M}} \tfrac{1}{i^{1+6\delta}}\, \right)^{\half}\\
\leq&\,C\,h^{\frac{r}{6}-r\delta}\,\left(\,1+\int_1^{\ssy{\sf M}}x^{-1-6\delta}\,dx\,\right)^{\half}\\
\leq&\,C\,h^{\frac{r}{6}-r\delta}\,\delta^{-\half}
\,\left(\,1-{\sf M}^{-6\delta}\,\right)^{\half},\\
\end{split}
\end{equation*}
which obviously yields \eqref{Lasso1} with $\epsilon=r\delta$.
\end{proof}
\subsection{Estimating the total error}
%
%
%
%
%
\begin{theorem}\label{FFQEWR}
Let $r=2$ or $3$, $\uu$ be the solution to the problem \eqref{AC2}, and
$(\UU_h^m)_{m=0}^{\ssy M}$ be the finite element approximations of $\uu$
constructed by \eqref{FullDE1}--\eqref{FullDE2}.
Then, there exists a constant ${\widehat c}_{\ssy{\rm TTL}}>0$,
independent of $h$, $\Delta\tau$, $\Delta{t}$ and ${\sf M}$, such that
\begin{equation}\label{all_all_all}
\max_{0\leq{m}\leq{\ssy M}}\left(\,{\mathbb E}\left[
\,\|\UU_h^m-\uu^m\|_{\ssy 0,D}^2\right]\,\right)^{\half}
\leq \,{\widehat c}_{\ssy{\rm TTL}}\,\left(\,\epsilon^{-\frac{1}{2}}_1\,\dtau^{\frac{1}{8}-\epsilon_1}
+\epsilon_2^{-\frac{1}{2}}\,\,\,h^{\frac{r}{6}-\epsilon_2}\,\right)
\end{equation}
for all $\epsilon_1\in\left(0,\tfrac{1}{8}\right]$ and
$\epsilon_2\in\left(0,\frac{r}{6}\right]$.
%
%
\end{theorem}
\begin{proof}
The error bound \eqref{all_all_all} follows easily from \eqref{ElPasso}, \eqref{Lasso1} and \eqref{split_split}.
\end{proof}
%
%
%
%
%
%
%
%
%
%
%
%
%
\section*{Acknowledgments}
Work partially supported by The Research Committee of The University of Crete under Research Grant \#4339:
`{\sl Numerical solution of stochastic partial differential equations}' funded by
The Research Account of the University of Crete (2015-2016).
%
%
%
%

%
\end{document}